\documentclass{amsart}
\usepackage{amsmath,amscd,amsfonts,amssymb,amsthm, graphicx, multirow, newlfont,exscale,xcolor,mathrsfs,stmaryrd, multicol, comment}

\title{A Jacobian criterion for nonsingularity \\ in mixed characteristic}
\author{Melvin Hochster}
\address{Department of Mathematics, East Hall, 530 Church St., Ann Arbor, MI 48109--1043, USA}
\curraddr{}
\email{hochster@umich.edu}

\author{Jack Jeffries}
\address{Department of Mathematics, Avery Hall,  1144 T St,
	Lincoln, NE 68588--0130, USA}
\curraddr{}
\email{jack.jeffries@unl.edu}

\date{\today}
\usepackage{enumitem}
\usepackage{bm}
\theoremstyle{plain}
\newtheorem{theorem}{Theorem}[section]
\newtheorem*{theorema}{Theorem A}
\newtheorem*{theoremb}{Theorem B}
\newtheorem*{theoremc}{Theorem C}

\newtheorem{corollary}[theorem]{Corollary}

\newtheorem{construction}[theorem]{Construction}
\newtheorem{proposition}[theorem]{Proposition}
\newtheorem{lemma}[theorem]{Lemma}

\theoremstyle{remark}
\newtheorem{remark}[theorem]{Remark}
\newtheorem{notation}[theorem]{Notation}
\newtheorem{discussion}[theorem]{Discussion}
\theoremstyle{definition}
\newtheorem{definition}[theorem]{Definition}
\newtheorem{example}[theorem]{Example}

\newcommand{\ben}{\begin{enumerate}}

\newcommand{\beni}{\ben[label=(\roman*)]}
\def\bsf{\boldsymbol{f}}

\newcommand{\Der}{\mathrm{Der}}
\newcommand{\Per}{\mathrm{Per}}

\newcommand{\inc}{\subseteq}

\newcommand{\cP}{\mathcal{P}}

\newcommand{\CM}{Cohen-Macaulay}

\newcommand{\inj}{\hookrightarrow}

\newcommand{\een}{\end{enumerate}}

\newcommand{\fa}{\mathfrak{a}}
\newcommand{\fb}{\mathfrak{b}}
\newcommand{\fB}{\mathfrak{B}}

\newcommand{\m}{\mathfrak{m}}
\newcommand{\n}{\mathfrak{n}}
\newcommand{\p}{\mathfrak{p}}
\newcommand{\q}{\mathfrak{q}}

\newcommand{\fra}{\mathrm{frac}}
\newcommand{\g}{{\gamma}}
\newcommand{\G}{{\Gamma}}
\newcommand{\Ge}{^{\G}_e}
\newcommand{\gt}{\wt{\g}}
\newcommand{\Gt}{^{\wt{\G}}}
\newcommand{\Gte}{\Gt_e}
\newcommand{\Ga}{^{\G}}

\newcommand{\height}{\textrm{height}}
\newcommand{\Hom}{\mathrm{Hom}}

\newcommand{\KGe}{K^{\G}_e}
\newcommand{\la}{{\lambda}}
\newcommand{\La}{{\Lambda}}

\newcommand{\N}{\mathbb{N}}

\newcommand{\Z}{\mathbb{Z}}

\newcommand{\sop}{system of parameters}

\newcommand{\Spec}{\textrm{Spec}}
\newcommand{\surj}{\twoheadrightarrow}
\newcommand{\surjdown}{\raisebox{.5pt}{$\downarrow$}\kern -5pt \raisebox{-1pt}{$\downarrow$}}

\newcommand{\tG}{\widetilde{\Gamma}}
\newcommand{\tGa}{^{\tG}}
\newcommand{\tL}{\widetilde{\Lambda}}

\newcommand{\ux}{\underline{x}}

\newcommand{\V}{\mathcal{V}}

\newcommand\vect[2]{#1_1,\,\ldots,\, #1_{#2}}

\newcommand\wh[1]{\widehat{#1}}
\newcommand\wt[1]{\widetilde{#1}}

\def\vect#1#2{{#1}_1, \, \ldots, \, {#1}_{#2}}

\newcommand{\mx}{\begin{pmatrix}}
	\newcommand{\emx}{\end{pmatrix}}

\def\todo#1
{\par\vskip 7 pt \noindent {\textcolor{yellow}\footnotesize \color{yellow} \fbox{\parbox{4.8in}{\color{black}
				\textbf{To do: } {\color{blue}#1}}}} \vskip 3 pt \par}
\def\forth#1
{\par\vskip 7 pt \noindent {\textcolor{yellow}\footnotesize \color{yellow} \fbox{\parbox{4.8in}{\color{black}
				\textbf{For thought: } {\color{blue}#1}}}} \vskip 3 pt \par}
			
\def\fifth#1
			{\par\vskip 7 pt \noindent {\textcolor{yellow}\footnotesize \color{yellow} \fbox{\parbox{4.8in}{\color{black}
							\textbf{To discuss: } {\color{ora}#1}}}} \vskip 3 pt \par}

\definecolor{grn}{rgb}{0.0, 0.5, 0.0}
\definecolor{prp}{rgb}{0.4, 0.0, 0.3}
\definecolor{ora}{rgb}{0.8, 0.4, 0.0}

\begin{document}
	
	\begin{abstract} We give a version of the usual Jacobian characterization of the defining ideal of the
	singular locus in the equal characteristic case:  the new theorem is valid for essentially affine algebras  over a complete local algebra over a mixed characteristic discrete valuation ring. The result makes use of the minors of a matrix that includes a row coming from the values of a $p$-derivation. To study the analogue of modules of differentials associated with the mixed Jacobian matrices that arise in our context, we introduce
	and investigate the notion of a {\it perivation}, which may be thought of, roughly, as a linearization of the notion of $p$-derivation. We also develop a mixed characteristic analogue of the positive characteristic $\Gamma$-construction, and apply this to give additional nonsingularity criteria. \end{abstract}
	
	\subjclass[2000]{Primary 13}
	

	\thanks{$^1$The first author was partially supported by National Science Foundation grants
		DMS--1401384, DMS--1902116, and DMS--2200501}
	\thanks{$^2$The second author was partially supported by National Science Foundation grant DMS--1606353 and CAREER Award DMS--2044833}
	
	\maketitle
	
	\pagestyle{myheadings}
	\markboth{MELVIN HOCHSTER AND JACK JEFFRIES}{A JACOBIAN CRITERION IN MIXED CHARACTERISTIC}
	
	\section{Introduction}\label{sec:intro}                

	The Jacobian criterion is a familiar tool for finding the singular locus of a finitely generated algebra over a perfect field: for an equidimensional variety, the singular locus is cut out by minors of the Jacobian matrix whose entries are partial derivatives of the defining equations. There is a long history in the literature of Jacobian criteria to compute the singular locus for various classes of algebras over fields, including \cite{Zar,Sam,Nag,BR} (see also the references of \cite{BR}); these are based on including additional derivations along with the partial derivatives of the variables. For example \cite{BR} gives criteria, in terms of derivations,  for $A_\p$ to be a formally smooth $k$-algebra  
in the $\p$-adic topology when $k$ is an $F$-finite field,  $A = S/I$, $S$ is a polynomial ring over a formally 
smooth, local, complete $k$-algebra $R$ and $\p \in \Spec(S)$ with $I \inc \p$.  The proof of this result depends on a 
regularity criterion for $A$ when $R$ is also assumed regular. 
	
	
	In this paper, we give a Jacobian type criterion for nonsingularity in mixed characteristic. One notable feature of this result is that, in contrast to equal characteristic, nonsingularity cannot be characterized in terms of smoothness over some base. Our criterion describes the singular locus in terms of minors of matrices involving partial derivatives and $p$-derivations (in the sense of Buium \cite{Bui} and Joyal \cite{Joy}) of the generators. We recall the definition of $p$-derivation in Section~\ref{sec:p-derivation} below; as an example, if $V$ is the $p$-adic integers, and 
$T=V[\vect xn]$  or $V[[\vect xn]]$, the map $\delta$ given by \[f(x_1,\dots,x_n) \mapsto  (f(x_1^p,\dots,x_n^p) - f(x_1,\dots,x_n)^p)/p\] is a $p$-derivation. 

	\begin{theorema}[Theorem~\ref{thm-A}]\label{thm:A}
		Let $(V,pV,K)$ be a discrete valuation ring with uniformizer $p\in \Z$. Let $T$ be a polynomial ring  over a power series ring over $V$ in $n$ variables (polynomial and power series variables) $\vect xn$ total.  Let $R:=T/I$ for some ideal $I = (\vect f a)T$ of pure height~$h$. Then, for a prime $\q$  of $R$ containing $p$, the local ring $R_\q$ is nonsingular if and only if $\q$ does not contain the $h\times h$ minors of the matrix with columns
		\begin{itemize} 
			\item $[\delta(f_1),\dots,\delta(f_a)]^T$, for some $p$-derivation mod $p^2$ on $V$ extended to $T$,
			\item $\left[\left(\frac{\partial f_1}{\partial x_i}\right)^p,\dots, \left(\frac{\partial f_a}{\partial x_i}\right)^p\right]^T$, for $i=1,\dots,n$,
			\item  $\left[\left(\frac{\partial f_1}{\partial \lambda}\right)^p,\dots, \left(\frac{\partial f_a}{\partial \lambda}\right)^p\right]^T$, for $\lambda\in \Lambda$ for a $p$-base $\Lambda$ of $K$.
			\end{itemize}
	\end{theorema}
See \ref{SS-insep} below for the notation $\frac{\partial f}{\partial \lambda}$ for $\lambda$ in a $p$-base. Note that for primes that do not contain $p$, $R_\q$ is nonsingular if and only if $\q$ does not contain the $h\times h$ minors of the usual Jacobian matrix; see Theorem~\ref{thm-A}. We remark also that any choice of $\delta$, any choice of $\Lambda$, and any choice of generators $f_1,\dots,f_a$ for $I$, and any presentation $R\cong T/I$ with $T$ as above is valid for the result above.


		The use of $p$-derivations to study singularities is motivated by \cite{DGJ}, where $p$-derivations were used to extend a classical theorem of Zariski and Nagata on symbolic powers. In fact, Theorem~A in the special case of a principal ideal $I$ follows from the main result of \cite{DGJ}. We note that, as our statement of Theorem~\ref{thm:A} is based on $p$-derivations, which generally do not exist in equal characteristic, there is no completely characteristic-free statement that subsumes this.

	
	The cokernel of the classical Jacobian matrix is just the module of K\"ahler differentials. In the case of algebras over a perfect field, the Jacobian criterion characterizes the regular locus as the locus on which the module of K\"ahler differentials is free of the ``correct'' rank. Theorem~A admits a similar interpretation along these lines. Given a ring $A$ with a $p$-derivation $\delta$ and an $A$-algebra $R$, we define a module $\widetilde{\Omega}_{R|A}$; in the situation of Theorem~A, this is just the cokernel of the matrix that appears in the statement. This module represents a functor of maps that satisfy simple functional identities analogous to derivations or $p$-derivations; we call these maps \emph{perivations}, and the module $\widetilde{\Omega}_{R|A}$ the \emph{universal perivation module}. In this language, we obtain the following version of Theorem~A.

\begin{theoremb}[Theorem~\ref{thm-B}]\label{thm:B}
			Let $V$ be a discrete valuation ring with uniformizer ${p\in \Z}$ and F-finite residue field (i.e., the Frobenius map on $V/pV$ is of finite index). Let $(R,\m,k)$ be  a local ring with $p\in \m$ that is essentially of finite type over a complete $V$-algebra. 
		
		The local ring $R$ is regular if and only if $\widetilde{\Omega}_{R|\Z}$ is free of rank $\dim(R)+\log_p[k:k^p]$.
	\end{theoremb}

In equal characteristic $p>0$, the notion of $\Gamma$-construction allows one to eliminate F-finiteness hypotheses from various statements; for an algebra $R$ essentially of finite type over a complete local ring, there is an F-finite faithfully flat purely inseparable extension $R\Ga$ that preserves many of the properties of $R$. In Section~\ref{sec:Gamma} below, we recall key properties of the $\Gamma$-construction, and develop a mixed characteristic analogue that we call the $\wt{\Gamma}$-construction with many of the same properties. Using the $\wt{\Gamma}$-construction, we state a version of Theorem~B without any F-finiteness hypothesis.

\begin{theoremc}
	[Corollary~\ref{cor-thmC}]\label{thm:C}
	Let $V$ be a discrete valuation ring with uniformizer ${p\in \Z}$. Let $(R,\m,k)$ a local ring with $p\in \m$ that is essentially of finite type over a complete $V$-algebra. Let $\wt{\Lambda}$ be a lift to $V$ of a $p$-base for $V/pV$. 
	
	The local ring $R$ is regular if and only if, for sufficiently small cofinite $\wt{\Gamma}\subseteq \wt{\Lambda}$, $\widetilde{\Omega}_{R\tGa|\Z}$ is free of rank $\dim(R)+\log_p[k(R^{\wt{\Gamma}}):k(R^{\wt{\Gamma}})^p]$, where $k(R^{\wt{\Gamma}})$ denotes the residue field of $R^{\wt{\Gamma}}$.
\end{theoremc}

Note that the equal characteristic Jacobian criteria of \cite{BR} and \cite{Nag} are also phrased in terms of cofinite subfields.
		
			A preliminary version of these results, namely Theorem~A in the case of finitely generated unmixed algebras over complete unramified discrete valuation rings with perfect residue fields, was announced in the MSRI Fellowship of the Ring Seminar in May 2020. Our original method was based on the argument in Subsection~\ref{subsec:elem}. These results were extended to the complete case and the essentially-of-finite-type-over-complete case by a N\'eron-Popsecu desingularization \cite{AR,Pop,Swan} argument in an earlier version of this manuscript, along with a version of Theorem~B in the corresponding cases. In the meanwhile, we learned of the work of Saito \cite{Sai}, who established the analogue of Theorem~B with the additional hypotheses that $R/pR$ is essentially of finite type over an F-finite field and that $R$ is flat over $\Z_{(p)}$. After learning of Saito's results, we used some of his ideas to give a different argument for Theorems~A and B that does not require any hypotheses on $V/pV$, and that does not require the use of N\'eron-Popsecu desingularization to address the cases where $R$ is complete or essentially of finite type over a complete ring. We have kept the original argument for the affine case here in Subsection~\ref{subsec:elem} for heuristic reasons.
	
	 We note that the universal perivation modules defined here appear as modules of FW-differentials in \cite{Sai}, as total $p^2$-differentials in \cite{DKRZ}, and implicitly as an extension class in \cite[\S 4.2]{Zda}; they are also closely related to the construction in \cite[9.6.12]{GabRom}.

	 

	\section{p-derivations and perivations}\label{sec:p-derivation}                         

	\subsection{$p$-derivations}
	
We recall first the notions of $p$-derivation and $\delta$-ring, in the sense of Buium \cite{Bui} and Joyal \cite{Joy}.

\begin{notation}
	For a prime integer $p$, we set
	\[ C_p(X,Y):= \frac{X^p + Y^p - (X+Y)^p}{p}= -\sum_{i=1}^{p-1} \frac{\binom{p}{i}}{p} X^i Y^{p-i} \in \Z[X,Y]. \]
		More generally, we set
	\begin{align*} C_p(X_1,\dots,X_t) :=& \ \frac{ X_1^p + \cdots + X_t^p - (X_1+\cdots+X_t)^p}{p} \\ =& \ -\hspace{-5mm}\sum_{\substack{a_1+\cdots+a_t=p \\ a_1,\dots,a_t\neq p}} \frac{\binom{p}{a_1,\dots,a_t}}{p} X_1^{a_1} \cdots X_t^{a_t} \in \Z[X_1,\dots,X_t]. \end{align*}
	For all $p$, we have 
	\begin{equation}\label{eq:Cpideal}\tag{$\clubsuit$}
	C_p(X,Y)\in (XY) \Z[X,Y] \ \text{and} \
C_p(X_1,\dots,X_t)\in (X_1,\dots,X_t)^2  \, \Z[X_1,\dots,X_t]. \end{equation}
\end{notation}

\begin{definition}\label{def:p-deriv}
	Let $R$ be a ring, and $p$ a prime integer. A map $\delta:R\to R$ is a \emph{$p$-derivation} if, for all $x,y\in R$, we have
	\begin{enumerate}
		\item $\delta(0)=\delta(1)=0$,
		\item $\delta(x+y)=\delta(x)+\delta(y)+C_p(x,y)$, and
		\item $\delta(xy)=x^p \delta(y) + y^p \delta(x) + p \delta(x)\delta(y)$.
	\end{enumerate}
	
	A \emph{$\delta$-ring} is a pair $(R,\delta)$, where $R$ is a ring, and $\delta$ is a $p$-derivation on $R$.
\end{definition}

Note that if $\delta$ is a $p$-derivation, then 
\begin{equation}\label{eq:pderadd}\tag{$\diamondsuit$}\delta(x_1+\cdots+x_t)=\delta(x_1)+\cdots+\delta(x_t)+C_p(x_1,\dots,x_t).\end{equation}

If $(R,\delta)$ is a $\delta$-ring, then the map
\[ \Phi:R\to R \qquad  \Phi(r)=r^p+p\delta(r) \]
is a ring homomorphism, and if $p\in R$ is not a unit, then the diagram
\[
\CD
R @>\Phi>> R \\
@VVV @VVV\\
R/pR @>F>> R/pR
\endCD
\]
commutes, where $F$ is the Frobenius map on $R/pR$; we call the map $\Phi$ the \emph{lift of Frobenius} on $(R,\delta)$. Conversely, if $p$ is a nonzerodivisor on $R$, and $\Phi$ is a ring endomorphism of $R$ for which the previous diagram commutes, then the map
\[ \delta:R \to R \qquad \delta(r)=\frac{\Phi(r)-r^p}{p} \]
is a $p$-derivation.

We note a few rings on which $p$-derivations exist (cf., \cite[Proposition~2.6]{DGJ}):
\begin{example}\label{ex:2.3}  
\begin{enumerate}
	\item On $\mathbb{Z}$, for any $p$, the identity map is the unique lift of Frobenius, and $\delta(n)=\frac{n-n^p}{p}$ is the unique $p$-derivation.
	\item Similarly, on the ring of $p$-adic integers $V=\widehat{\mathbb{Z}_{(p)}}$, the identity map is the unique lift of Frobenius, and 
	$\delta(n)=\frac{n-n^p}{p}$ is the unique $p$-derivation.
	\item Generalizing the last example, the ring of Witt vectors $W(k)$ over a field $k$ of positive characteristic admits a $p$-derivation.
	\item Any lift of Frobenius $\Phi$ on a ring $R$ extends to a lift of Frobenius $\tilde{\Phi}$ on $R[x_1,\dots,x_n]$ by setting $\tilde{\Phi}(x_i)=x_i^p + p f_i$ for arbitrary elements ${f_i\in R[x_1,\dots,x_n]}$.
	\item Any lift of Frobenius $\Phi$ on a $p$-adically complete ring $R$ extends to a lift of Frobenius $\tilde{\Phi}$ on $R[[x_1,\dots,x_n]]$ by setting $\tilde{\Phi}(x_i) :=x_i^p + p f_i$ for arbitrary elements $f_i\in R[[x_1,\dots,x_n]]$, since $R[[\vect x n]]$ is complete with respect to $(p,\vect xn)$.
\end{enumerate}
\end{example}


\begin{remark}\label{remark:der.prod} Let $\delta$ be  a $p$-derivation on $R$. 

\begin{enumerate}

\item\label{item:der.prod1} If $\delta$ maps a set of generators of an ideal $\fa$ of $R$ into $\fa' \supseteq \fa$,  then $\delta$ maps
 $\fa$  into $\fa'$. 
 
 \item\label{item:der.prod2} If $\fa$ and $\fb$ are ideals of $R$ such that $\delta$ maps $\fa$ into $\fa' \supseteq \fa$ and $\delta$ maps $\fb$ into $\fb' \supseteq \fb$ then 
 $\delta$ maps $\fa \fb$ into $\fa' \fb'$.
 
 \item\label{item:der.prod3} If $\delta$ maps a set of generators of an ideal $\fa$ of $R$ into $\fa' \supseteq \fa$, then $\delta$ maps 
 $\fa^k$ into ${\fa'}^k$ for every
 positive integer $k$.  

 \item\label{item:der.prod4} If $\fa$ is an ideal containing $p$,  then $\delta$ maps $\fa^{k+1}$ into $\fa^{k}$ for all $k \in \N$.

 \item\label{item:der.prod5} If $\fa \inc \fa'$ are ideals with $\delta(\fa) \inc \fa'$ and $r \equiv s$ modulo $\fa$,  then 
 $\delta(r) \equiv \delta(s)$ modulo $\fa'$.
 
 \item\label{item:der.prod6}  If $p \in \fa$,  $k \in \N$, and $r \equiv s$ modulo $\fa^{k+1}$, then $\delta(r) \equiv \delta(s)$ mod $\fa^k$.  

   \item\label{item:der.prod7} For any element $r \in R$,  $\delta(pr) \equiv  r^{p}$ modulo $pR$.
 
 \item\label{item:der.prod8} For  
 $\vect r n \in R$,  $\delta(r_1 \cdots r_n) \equiv \sum_{i=1}^n (\prod_{j\not= i} r_j^{p})\delta(r_i)$ modulo $pR$.  Hence,
 for any $r \in R$,  $\delta(r^p) \equiv 0$ modulo $pR$.  
 
 \end{enumerate}
 
  \end{remark}
  \begin{proof}
  Parts (1) and (2) are straightforward consequences of \eqref{eq:Cpideal}, \eqref{eq:pderadd}, and part (3) of Definition~\ref{def:p-deriv}. Part (3) follows from (1) and (2) by induction. For (4), we induce on $k$, with the base case $k=1$ trivial. Take a generating set $a_1,\dots,a_t$ for $\fa$ and a generating set $c_1,\dots,c_s$ for $\fa^k$, so $\fa^{k+1}$ is generated by elements of the form $a_i c_j$. Then $\delta(a_i c_j) = a_i^p \delta(c_j) + c_j^p \delta(a_i) +p\delta(a_i)\delta(c_j)\in \fa^k$ by the induction hypothesis. Part (5) is immediate from part (1) of Definition~\ref{def:p-deriv} and \eqref{eq:Cpideal}. Part (6) follows from (4) and (5). Part (7) follows from part (3) of Definition~\ref{def:p-deriv} and $\delta(p)=1-p^{p-1}$. Part (8) follows from  part (3) of Definition~\ref{def:p-deriv} by a straightforward induction on~$n$.  \end{proof}


We refer the reader to \cite[\S 2]{BS} for some basic properties of $\delta$-rings, and \cite[\S 3]{BS} for connections between the theory of $\delta$-rings and perfectoid spaces.

\subsection{Derivations of inseparability}\label{SS-insep}

We briefly recall some basics of derivations in positive characteristic. We refer the reader to \cite[Section~36]{Mat} for the facts below. A \textit{$p$-base} of a field $K$ of positive characteristic is a set of elements $\Lambda\subset K$ such that $\{ \lambda_1 ^{a_1} \cdots \lambda_t^{a_t} \ | \ \lambda_1,\dots,\lambda_t \in \Lambda, 0\leq a_1,\dots,a_t < p\}$ forms a $K^p$-basis for $K$. Such a set exists, by an application of Zorn's Lemma, and the cardinalities of any two $p$-bases for a field $K$ are equal.  

Given a $p$-base $\Lambda$ for a field $K$, one obtains a derivation for each element $\lambda_0\in \Lambda$: explicitly, $\frac{\partial}{\partial \gamma_{\lambda_0}}$ is the $K^p$-linear map that maps $\lambda_0^{a_0} \lambda_1^{a_1} \cdots \lambda_t^{a_t}$ to $a_0 \lambda_0^{a_0-1} \lambda_1^{a_1} \cdots \lambda_t^{a_t}$ for all $\lambda_1,\dots,\lambda_t \in \Lambda\smallsetminus\{\lambda_0\}$, $0\leq a_0,a_1,\dots,a_t < p$.

The module of differentials of $K$ has a $K$-vector space basis $\{d\lambda \ | \ \lambda\in \Lambda\}$ such that the dual element in $\Hom_K(\Omega_K,K)$ to $d\lambda_0$ is $\frac{\partial}{\partial \gamma_{\lambda_0}}$.
	
	\subsection{$p$-derivations modulo $p^2$}
	
	\begin{definition}
		We say that a map $\delta:R/p^2 R \to R/pR$ is a \emph{$p$-derivation mod $p^2$} on $R$ if $\delta$ satisfies the axioms (1)-(3) of Definition~\ref{def:p-deriv}.
		
		A \textit{mod $p^2$ $\delta$-ring} is a pair $(A,\delta)$, where $\delta$ is a $p$-derivation mod $p^2$ on $A$.
	\end{definition}

We note that the conclusions of Remark~\ref{remark:der.prod} 
also hold for a $p$-derivation mod $p^2$, since we may apply these statements to $R/p^2 R$.
	
If $R/p^2 R$ is flat over $\Z/p^2 \Z$, e.g., if $p$ is a regular element on $R$, then $p$-derivations mod $p^2$ are in bijection with lifts of Frobenius to $R/p^2R$.
	By Remark~\ref{remark:der.prod}, any $p$-derivation induces a $p$-derivation mod $p^2$. However, existence of a $p$-derivation mod $p^2$ is less stringent than existence of a $p$-derivation. This follows from the following lemma and subsequent example.
	
	\begin{lemma}
		Let $(A,\delta)$ be a $\delta$-ring, with $p$ a regular element on $A$, and let $S$ be a smooth $A$-algebra.  Then there exists a $p$-derivation mod $p^2$ on $S$ that extends $\delta$.
	\end{lemma}
	\begin{proof}
		Let $\Psi$ be the lift of Frobenius associated to $A$. There is a commutative diagram
		\[
		\CD
		S @>>> S/pS @>{F}>> S/pS \\
		@AAA @. @AAA  \\
		A @>{\Psi}>> A @>>> S/p^2S. \\
		\endCD
		\]
		By smoothness, there is a ring homomorphism $S\to S/p^2 S$ that makes the diagram commute; such a map is a lift of Frobenius mod $p^2$.
	\end{proof}
	
	On the other hand, we note:

		\begin{example}
			Let \[V=\widehat{\mathbb{Z}_{(p)}}, \ p \not= 2,\,3, \ A=V[y] ,  \ \text{and} \ 
			R=\frac{A[x]_{2x+y}}{(x^2 + yx + p)}.\] Note that $R$ is a standard \'etale extension of~$A$.
			To extend the $p$-derivation $\delta$ on~$A$ that sends $y$ to $0$, we need $g\in R$  
			  such that $x\mapsto x^p + pg, y\mapsto y^p$ yields a $V$-algebra homomorphism.
			For this, \ we need  $(x^p + pg)^2 + y^p(x^p + pg) +p = 0$  in $R$.
That is, $g$ must be a solution of	\[p^2g^2 +(2px^p + py^p)g  + (x^{2p} + y^px^p + p) = 0\]
			in $R$. In the fraction field of $R$, we may apply the quadratic formula to obtain
			\[g = \frac{-(2px^p + py^p) \pm p\sqrt{y^{2p} - 4p}  }{2p^2}
                 	= \frac{-(2x^p + y^p) \pm \sqrt{y^{2p} - 4p}}{2p}.\] 
Since $g$ must be in $R$,   $\sqrt{y^{2p} - 4p} \in R$.  We will show that this is false.  First note
that $B = A[x]/(x^2+yx+p) = A[\theta]$ where $\theta =(-y + \sigma)/2$  and $\sigma$ is
a square root of $y^2-4p$.  Then $B= A[\sigma]$, since 2 is a unit in $V$.  Since $A$ is a UFD and $y^2-4p$
is squarefree in $A$,  $B$ is normal.  Hence, if $y^{2p} - 4p$ has a square root in the fraction field of $B$ (which
contains $R$), it has a square root in $B$,  and this will have the form $a_0 + a_1\sigma$, with each $a_i \in A$. 
Since $B=A\oplus A\sigma$, it is clear that $(a_0 + a_1\sigma)^2$ is not in $B$ unless $a_0 = 0$ or $a_1=0$.  We cannot have $a_0 = 0$,
since  $y^{2p} - 4p$ is not a square in $A$.  We cannot have $a_0 = 0$, or  else  $y^{2p}-4p = a_1^2(y^2-4p)$.
This is a contradiction, since substituting $y^2 = 4p$  in $y^{2p}-4p$ does not yield 0: its value is
$(4p)^p-4p \neq 0$.  We conclude that no such $g$ exists, and hence no such extension of $\delta$ exists.
		\end{example}
	
We now aim to show that any discrete valuation ring with uniformizer $p$ admits a $p$-derivation modulo $p^2$. Our construction is a variation on Saito \cite[Lemma~1.3.1]{Sai} in the case of a perfect residue field. We require a lemma first.
	
	\begin{lemma}
		For any ring $R$, prime integer $p$, and elements $x,y\in R$, \[(x+y)^{p^2} + p C_p(x,y)^p \equiv x^{p^2} + y^{p^2} \ \mathrm{mod} \ p^2R.\]
	\end{lemma}
\begin{proof}
	We have $(x+y)^{p^2}= \sum_{i+j=p^2} \binom{p^2}{i} x^i y^j$. By Kummer's Theorem \cite{Kum} $p^2$ divides $\binom{p^2}{i}$ unless $p$ divides $i$, so we have  $(x+y)^{p^2} \equiv \sum_{i+j=p} \binom{p^2}{pi} x^{pi} y^{pj} \ \mathrm{mod} \ p^2R.$ We have that  $\binom{pn}{pm} \equiv \binom{n}{m} \mathrm{mod} \ p^2 \Z$, (cf. \cite[Exercise~1.6(c)]{Sta}), so 
	\begin{align*}(x+y)^{p^2} &\equiv \sum_{i+j=p} \binom{p}{i} x^{pi} y^{pj} \\
	&\equiv x^{p^2} + y^{p^2} - p C_p(x^p,y^p)\end{align*} modulo $p^2R.$
\end{proof}
		
		\begin{proposition}\label{prop:p-der-any-V}
			Let $(V,pV,K)$ be a discrete valuation ring with uniformizer $p$. Let $\{\gamma_{\lambda} \ | \ \lambda\in \Lambda\}$ be a set of elements in $V$ that maps bijectively to a $p$-base for $K$. Then the map
			\[ \psi \left( \sum_{\substack{\alpha\in \Z^{\oplus\Lambda}\\ 0\leq \alpha_{\lambda}<p}} a_{\alpha}^p \gamma^{\alpha} +p b \right) = \sum_{\substack{\alpha\in \Z^{\oplus\Lambda}\\ 0\leq \alpha_{\lambda}<p}} a_{\alpha}^{p^2} \gamma^{p\alpha} +p b^p \]
			is a lift of Frobenius modulo $p^2$ on $V$. Hence, the map $\delta: V\to K$ given by $\delta(f)=\frac{\psi(f)-f^p}{p}$ is a $p$-derivation modulo $p^2$ on $V$.
		\end{proposition}
	\begin{proof}
		First, we show that the map $\psi$ is well-defined. We have that \[V/pV = \bigoplus_{{\alpha\in \Z^{\Lambda}, 0\leq \alpha_{\lambda}<p}} (V/pV)^p \overline{\gamma^{\alpha}},\] so any element in $V$ can be written as a finite sum $\sum a_{\alpha}^p \gamma^{\alpha} +p b$ for some elements $a_{\alpha},b\in V$. If 
			\[\sum_{\substack{\alpha\in \Z^{\oplus\Lambda}\\ 0\leq \alpha_{\lambda}<p}} a_{\alpha}^p \gamma^{\alpha} +p b = \sum_{\substack{\alpha\in \Z^{\oplus\Lambda}\\ 0\leq \alpha_{\lambda}<p}} c_{\alpha}^p \gamma^{\alpha} +p d,\]
			then $\sum_{{\alpha\in \Z^{\oplus\Lambda}, 0\leq \alpha_{\lambda}<p}} (a_{\alpha}^p - c_{\alpha}^p) \gamma^{\alpha} \in pV$, so each $a_{\alpha}^p - c_{\alpha}^p\in pV$ by $K^p$-linear independence of the images of the elements $\gamma_{\lambda}$, and hence $a_{\alpha} - c_{\alpha} \in pV$ by injectivity of the Frobenius on $K$. Substituting back, we have that 
			\[p(b-d)\in \sum_{{\alpha\in \Z^{\oplus\Lambda}, 0\leq \alpha_{\lambda}<p}} ((a_{\alpha}^p + p e_\alpha)^p - a_{\alpha}^p) \gamma^{\alpha}\] for some elements $e_{\alpha}\in V$, and hence this is in $p^2 V$. Since $p$ is a regular element on $V$, $b-d\in pV$ as well. Well-definedness of the map is then evident.
			
			We check now that $\psi$ is additive. We have
			\begin{align*}&\psi\left(\sum a_{\alpha}^p \gamma^{\alpha} +p b + \sum c_{\alpha}^p \gamma^{\alpha} +p d \right)\\ &= 
			\psi\left( \sum \left((a_{\alpha} + c_{\alpha})^p + p C_p(a_{\alpha},c_{\alpha})\right) \gamma^{\alpha} + p (b+d) \right)\\
			&= 	\psi\left( \sum (a_{\alpha} + c_{\alpha})^p \gamma^{\alpha} + p (b+d + \sum C_p(a_{\alpha}, c_{\alpha})\gamma^{\alpha} ) \right)
		\\
		&=\sum (a_\alpha + c_\alpha)^{p^2} \gamma^{p\alpha} + p (b+d+ \sum C_p(a_\alpha,c_\alpha) \gamma^\alpha)^p\\
		&\equiv \sum (a_{\alpha}^{p^2} + c_{\alpha}^{p^2})\gamma^{p\alpha} + p (b^p +d^p)\\
		&= \psi\left(\sum a_{\alpha}^p \gamma^{\alpha} +p b \right) + \psi\left( \sum c_{\alpha}^p \gamma^{\alpha} +p d \right),
			\end{align*}
			where the equivalence in the penultimate line is modulo $p^2V$, using the previous lemma and the fact that $p(X+Y+Z)^p \equiv p(X^p+Y^p+Z^p)$ modulo $p^2$.
			
			The verification of multiplicativity modulo $p^2 V$ is straightforward, and omitted.
	\end{proof}

Combining this with Example~\ref{ex:2.3}, we obtain:

\begin{corollary}\label{cor:p-der-anythingoveranyV}
	Let $V$ be a discrete valuation ring with uniformizer $p\in \Z$. Any polynomial ring in arbitrarily many variables, power series ring in finitely many variables, or polynomial ring over a power series ring in finitely many variables over $V$ admits a $p$-derivation mod~$p^2$. Explicitly, one may take a $p$-derivation mod~$p^2$ as constructed in Proposition~\ref{prop:p-der-any-V} and extend by prescribing arbitrary values to the variables.
\end{corollary}
 Indeed, by parts~(4) and~(5) of Example~\ref{ex:2.3} and the previous proposition, there is a $p$-derivation modulo $p^2$ on $\widehat{V}[[x]][y]$, for a finite set of variables $x$ and an arbitrary set of variables $y$. Since we have canonical isomorphisms $\widehat{V}[[x]][y]/p^i \widehat{V}[[x]][y] \cong {V}[[x]][y]/p^i {V}[[x]][y]$ for $i=1,2$, such a map can be identified with a $p$-derivation modulo $p^2$ on ${V}[[x]][y]$.

\subsection{Perivations}

In this subsection, we discuss the notion of \emph{perivation}, which is essentially the same as that of \emph{total $p$-derivations} in the sense of Dupuy, Katz, Rabinoff, and Zureick-Brown \cite{DKRZ}, and the notion of FW-derivation of Saito \cite{Sai}.

\begin{definition}[cf. {\cite[Definition~2.1.1]{DKRZ}}]\label{def:peri}
	Let $R$ be a ring, and $p$ a prime integer. Let $M$ be an $R/pR$-module. A map $\alpha:R \to M$ is a \emph{perivation} if for all $x,y\in R$, we have
	\begin{enumerate}
		\item $\alpha(0)=\alpha(1)=0$,
		\item $\alpha(x+y)=\alpha(x)+\alpha(y)+C_p(x,y) \alpha(p)$, and
		\item $\alpha(xy)=x^p \alpha(y) + y^p \alpha(x)$.
	\end{enumerate}
	We call $\alpha(p)$ the \emph{distinguished element} of $\alpha$.
	If additionally, $(A,\delta)$ is a mod $p^2$ $\delta$-ring, with $A$ a subring of $R$, and
	\begin{enumerate}\setcounter{enumi}{3}
		\item\label{condition-standard} $\alpha(a)=\delta(a)\alpha(p)$ \ for all $a\in A$,
	\end{enumerate}
	then we say $\alpha$ is a \emph{perivation over $A$}.
\end{definition}

\begin{example}
	If $(R,\delta)$ is a mod $p^2$  $\delta$-ring, then the composition $R\xrightarrow{\delta} R \twoheadrightarrow R/pR$ is a perivation with distinguished element $1$. More generally, if $M$ is an $R/pR$-module, and $t\in M$, then the map $R\xrightarrow{\delta} R \xrightarrow{\cdot t} M$ is a perivation with distinguished element~$t$;  in this setting, we call a perivation of this form \emph{trivial}. Thus, condition~(\ref{condition-standard}) above says that the restriction of $\alpha$ to $A$ is trivial.
\end{example}

\begin{example}
	Given a ring $R$, consider the ring homomorphism $\beta$ given as the composition $R \twoheadrightarrow R/pR \xrightarrow{F} R/pR$,  where $F$ is the Frobenius map. If $R/pR$ is considered as an $R$-algebra via $\beta$, then any derivation of $R$ into an $R/pR$-module is a perivation with distinguished element zero. Conversely, a perivation to an $(R/pR)$-module with distinguished element zero is a derivation (with respect to $\beta$). We call such a map a \emph{derivation of the Frobenius} for short. We write $\Der^F_{R|A}(M)$ for the collection of $A$-linear derivations of the Frobenius from $R$ into an $(R/pR)$-module $M$.
	
	If $R$ has characteristic $p$, then for any perivation, we have $\alpha(p)=\alpha(0)=0$, so a perivation is just a derivation of the Frobenius.
\end{example}

\begin{remark}
	Any perivation on a ring $R$ is a perivation over $\Z$. Indeed, the unique $p$-derivation on $\Z$ is given by the rule $\delta(n)=(n-n^p)/p$, and by a simple induction, we have $\alpha(n)=n \alpha(1) + C_p(1,\dots,1)\alpha(p)=(n-n^p)/p \cdot \alpha(p)$. Thus, we lose no generality by restriction to the relative setting of perivations over a $\delta$-ring; we may take $\Z$ as the base ring.
\end{remark}

We record a few basic observations on perivations.
\begin{remark}\begin{enumerate}
\item	If $M$ is an $R/pR$-module, then the collection of perivations from $R$ to $M$ admits a natural $R/pR$-module structure by postmultiplication, and likewise for perivations over a mod $p^2$ $\delta$-ring $A$. We denote these modules by $\Per_R(M)$ and $\Per_{R|(A,\delta)}(M)$, respectively.
\item If $\alpha:R\to M$ is a perivation, and $\phi:M\to N$ is an $R$-module homomorphism, then $\phi\circ \alpha$ is a perivation, as is readily verified from the definition.
\end{enumerate}
\end{remark}

\begin{discussion}[$p$-linear maps and the Peskine-Szpiro functor]\label{p-linear}
	Recall that if $R$ is a ring of prime characteristic $p>0$, a map $\eta$ between two $R$-modules $M$ and $N$ is {\emph{$p$-linear}} if it is additive and $\eta(rm)=r^p \eta(m)$ for all $r\in R$ and $m\in M$. If $F:R\to R$ is the Frobenius map on $R$, let ${}^{1}(-):R\mathrm{-mod} \to R\mathrm{-mod}$ be the functor of restriction of scalars through $F$, and $F_R(-):R\mathrm{-mod} \to R\mathrm{-mod}$ be the functor of extension of scalars through $F$.
	
	A map $\eta:M\to N$ is $p$-linear if and only $\eta:M \to {}^{1}N$ is $R$-linear. By Hom-tensor adjunction, there is a natural isomorphism  \begin{align*} \Hom_R(M,{}^{1}N) &\cong \Hom_R(F_R(M),N) \\ \eta &\mapsto \big((r \otimes m) \mapsto r \eta(m) \big);\end{align*}
	thus, we can identify $p$-linear maps from $M$ to $N$ with $R$-linear maps $F_R(M)\to N$. By abuse of notation, we will use the same name for maps we identify in this way.
\end{discussion}

\begin{lemma}\label{peri-factor-I2pI}
	Let $(A,\delta)$ be a mod $p^2$ $\delta$-ring, $R$ be an $A$-algebra, and $I$ be an ideal of $R$. Let $\alpha:R\to M$ be a perivation over $A$.
	\begin{enumerate}
		\item\label{lem-desc1} $\alpha$ descends to a well-defined perivation $\bar{\alpha}: R/(I^2+pI) \to M/IM$.
		\item\label{lem-desc2} The restriction of $\bar{\alpha}$ to $I/(I^2+pI) \to M/IM$ is a $p$-linear map over $R/pR$.
		\item\label{lem-desc3} If $\alpha(I)\subseteq IM$, then $\alpha$ descends to a perivation $\alpha': R/I \to M/IM$.
		\item\label{lem-desc4} If $IM=0$, then there is a natural bijection between perivations from $R/I$ to $M$ and perivations from $R$ to $M$ that map $I$ to $0$.
	\end{enumerate}
\end{lemma}
\begin{proof}
	For part (\ref{lem-desc1}), let $r\in R$, and $a_0,a_1,\dots,a_t,a'_1,\dots,a'_t\in I$. Then
	\begin{align*} &\alpha(r+pa_0+ \sum_i a_i a'_i) = \alpha(r)+\alpha(pa_0)+ \alpha(\sum_i a_i a'_i) + C_p(r,p a_0,\sum_i a_i a'_i)\alpha(p) 
	\\&= \alpha(r)+\alpha(pa_0)+ \sum_i \alpha( a_i a'_i) + C_p(r,p a_0,\sum_i a_i a'_i)\alpha(p) + C_p(a_1 a'_1,\dots,a_t a'_t) \alpha(p).
	\end{align*}
	Since $C_p(r,p a_0,\sum_i a_i a'_i)$ and $C_p(a_1 a'_1,\dots,a_t a'_t)$ lie in $I+(p)$,  we have, modulo $IM$,
	\begin{align*} \alpha(r+pa_0&+ \sum_i a_i a'_i) \equiv \alpha(r)+\alpha(pa_0)+ \sum_i \alpha(a_i a'_i) \\
	&\equiv \alpha(r)+p^p \alpha(a_0) + a_0^p \alpha(p) + \sum_i {a'}_i^p \alpha(a_i) + \sum_i a_i^p \alpha({a'}_i) \equiv \alpha(r).
	\end{align*}
	For part (\ref{lem-desc2}), we note that
	\[\alpha(r a_0)= r^p \alpha(a_0) + a_0^p \alpha(r) \equiv r^p \alpha(a_0) \quad \text{modulo} \ IM.\]
	Part (\ref{lem-desc3}) is clear. For part (\ref{lem-desc4}), we note first that given a perivation from $R/I$ to $M$, the map from $R$ to $M$ obtained by precomposing with the quotient map is a perivation from $R$ to $M$ that maps $I$ to $0$. Conversely, by part (\ref{lem-desc3}), a perivation from $R$ to an $R/I$-module that sends $I$ to $0$ factors through the quotient map. 
	\qedhere\end{proof}

In particular, any perivation $R\to M$ factors through $R/p^2R$. 

\section{Universal perivation modules}\label{sec:peri}

In this section, we study universal objects for perivations. 

\begin{definition}
Let $(A,\delta)$ be a mod $p^2$ $\delta$-ring, and $R$ be an $A$-algebra.
	We say that a perivation $\alpha:R\to M$ is a \emph{universal perivation over $A$} and $M$ is a \emph{universal perivation module over $A$} if for any perivation $\beta:R\to N$ over $A$ there is a unique $R$-module homomorphism $\phi:M\to N$ such that $\beta=\phi\circ \alpha$. We will write $\tilde{d}_{R|(A,\delta)}: R \to \widetilde{\Omega}_{R|(A,\delta)}$ for a universal perivation.
\end{definition}

In particular, for a universal perivation module $\widetilde{\Omega}_{R|(A,\delta)}$, there are natural isomorphisms $\Per_{R|(A,\delta)}(M) \cong \Hom_{R}(\widetilde{\Omega}_{R|(A,\delta)},M)$ for all $R$-modules $M$.

Universal perivations for $A=\Z$ were constructed in \cite{DKRZ} (where universal perivation modules are called \emph{total $p$-differentials}), and were used there to study existence of $p$-derivations on $\Z/p^2\Z$-algebras. In particular they always exist; a construction for these modules appears below. These modules appear in \cite{Sai} as modules of FW-differentials; many of the properties established below also appear there. They also appear implicitly in the work of Zdanowicz \cite[\S 4.2]{Zda} on existence of lifts of Frobenius / $p$-derivations on $\Z/p^2\Z$-algebras, and are related to a construction of Gabber and Romero  \cite[9.6.12]{GabRom}. 

In this section, we record some basic properties of universal perivations. Many of these results overlap with \cite{DKRZ} and \cite{Sai}.

By uniqueness of representing objects (i.e., Yoneda's Lemma), universal perivation modules are unique up to natural isomorphism. Moreover, if $\phi:R\to S$ is an $A$-algebra homomorphism, and $M$ is an $S/pS$-module, we obtain a map
\begin{align*} \Per_{S|(A,\delta)}(M) &\to \Per_{R|(A,\delta)}(M) \\ \alpha&\mapsto \alpha\circ \phi\end{align*}
which induces the natural map
\[ \widetilde{\Omega}_{R|(A,\delta)} \xrightarrow{\tilde{d}_\phi} \widetilde{\Omega}_{S|(A,\delta)}. \]

\begin{remark}\label{rem:charp} If $R$ has characteristic $p>0$, then $\Per_{R|(A,\delta)}(M) = \Der^F_{R/pR|A/pA}(M)$, and hence
$\widetilde{\Omega}_{R|(A,\delta)} \cong F_{R/pR}(\Omega_{R/pR|A/pA})$.
\end{remark}

We record some analogues of the fundamental sequences for differentials.

\begin{proposition}[First fundamental sequences]\label{fundl-1}
	\begin{enumerate}
		\item Let $(A,\delta)$ be a mod $p^2$ $\delta$-ring. Let $R\xrightarrow{\phi} S$ be an $A$-algebra homomorphism. Then there is a natural exact sequence
		\[ S\otimes_R \widetilde{\Omega}_{R|(A,\delta)} \to \widetilde{\Omega}_{S|(A,\delta)} \to F_{S/pS}(\Omega_{S/pS|R/pR}) \to 0. \]
		\item Let $(A,\delta) \to (B,\varepsilon)$ be a mod $p^2$ $\delta$-ring homomorphism (i.e., a ring homomorphism $\psi$ such that $\psi \circ \delta = \varepsilon \circ \psi$), and $B \to R$ be a ring homomorphism. Then there is a natural exact sequence
		\[ R\otimes_B F_{B/pB}({\Omega}_{B/pB|A/pA}) \to \widetilde{\Omega}_{R|(A,\delta)} \to \widetilde{\Omega}_{R|(B,\varepsilon)} \to 0. \]
	\end{enumerate}	
\end{proposition}
\begin{proof}
	\begin{enumerate}
		\item First, we claim that there are natural exact sequences
		\[ 0 \to \Der^F_{S/pS|R/pR}(M) \to  \Per_{S|(A,\delta)}(M) \to \Per_{R|(A,\delta)}(M) \]
		for $S/pS$-modules $M$. We have the first inclusion, since $\Der^F_{S/pS|R/pR}(M)$ is naturally isomorphic to $\Der^F_{S|R}(M)$, and the latter naturally embeds as a subset in $\Per_{S|(A,\delta)}(M)$. For exactness in the middle, we need to see that a perivation that vanishes on $R$ is an $R$-linear derivation of the Frobenius. For such a perivation, the distinguished element is $0$ (since $p\in R$), so it is a derivation of Frobenius, and the claim is then clear.
		
		Now, since \[\Hom_S(S\otimes_R \widetilde{\Omega}_{R|(A,\delta)} , M) \cong \Hom_R(\widetilde{\Omega}_{R|(A,\delta)},\Hom_S(S,M)) \cong \Per_{R|(A,\delta)}(M)\] and $\Hom_{S|(A,\delta)}(F_{S/pS}(\Omega_{S/pS|R/pR}),M)\cong \Der^F_{S|R}(M)$ for all $S/pS$-modules~$M$, by Yoneda's Lemma, there is a complex as in the statement, which must be exact.
		\item Along the same lines, it suffices to show that for every $R/pR$-module $M$ we have an exact sequence
		\[ 0 \to \Per_{R|(B,\varepsilon)}(M) \to  \Per_{R|(A,\delta)}(M) \to \Der^F_{B|A}(M). \]
		For the map $\Per_{R|(A,\delta)}(M) \to \Der^F_{B|A}(M)$, we send a perivation $\alpha$ to the map $\beta=\alpha-\alpha(p)\varepsilon$. One verifies immediately that this is an $A$-linear derivation of the Frobenius on $B$, and that $\beta$ is zero if and only if $\alpha$ restricted to $B$ is a trivial perivation.\qedhere
	\end{enumerate}
\end{proof}

\begin{remark}
	As a particular case of the first fundamental sequence, when $R=A$, we have 
	\[ S/pS \to \widetilde{\Omega}_{S|A} \to F_{S/pS}(\Omega_{S/pS|A/pA}) \to 0. \]
	This exact sequence was considered in \cite{DKRZ} and implicitly in \cite{Zda} in connection with existence of lifts of Frobenius / $p$-derivations modulo $p^2$. In particular, this sequence is split exact if and only if $\delta$ extends to a $p$-derivation mod $p^2$ on $R$: such a map is a splitting from $\widetilde{\Omega}_{S|(A,\delta)}\to S/pS$.
\end{remark}

\begin{proposition}[Second fundamental sequence]\label{fundl-2}
	Let $(A,\delta)$ be a mod $p^2$ $\delta$-ring. Let $R$ be an $A$-algebra, and $I$ an ideal of $R$. Then there is a natural exact sequence
	\[ F_{R/(I+pR)}(I/I^2+pI) \to R/I \otimes_R \widetilde{\Omega}_{R|(A,\delta)} \to \widetilde{\Omega}_{(R/I) | (A,\delta)} \to 0. \]
\end{proposition}
\begin{proof}
	For the first map above, we consider the perivation given as the composition:
	\[ R \xrightarrow{\tilde{d}_R} \widetilde{\Omega}_{R|(A,\delta)} \twoheadrightarrow R/I \otimes_R \widetilde{\Omega}_{R|(A,\delta)}; \]
	by Lemma~\ref{peri-factor-I2pI}~(\ref{lem-desc2}), this induces a $p$-linear map from $I/(I^2+pI)$ to $R/I \otimes_R \widetilde{\Omega}_{R|(A,\delta)}$, which by Discussion~\ref{p-linear}, yields the map above.
	
	Observe that if $M$ is an $R/(I+pR)$-module, then \[\Per_{R|(A,\delta)}(M)\cong \Hom_R(\widetilde{\Omega}_{R|(A,\delta)},M)\cong \Hom_R(R/I \otimes_R \widetilde{\Omega}_{R|(A,\delta)},M).\] Thus, along the same lines as the  proof of Proposition~\ref{fundl-1}, it suffices to show that we have natural exact sequences
	\[  0 \to \Per_{R/I |(A,\delta)}(M) \to \Per_{R|(A,\delta)}(M) \to \Hom_{R/I}(F_{R/(I+pR)}(I/(I^2+pI)),M)  \]
	for all $R/(I+pR)$-modules $M$. By Lemma~\ref{peri-factor-I2pI}~(\ref{lem-desc4}), any nonzero perivation of $R/I$ into $M$ pulls back to a nonzero perivation of $R$ into $M$, so exactness holds at the first nonzero spot. A perivation of $R$ into $M$ maps to zero in \[\Hom_{R/I}(F_{R/(I+pR)}(I/(I^2+pI)),M)\] exactly if its restriction to $I$ is zero; by Lemma~\ref{peri-factor-I2pI}~(\ref{lem-desc4}), this happens if and only if it is the image of a perivation from $R/I$ to~$M$.
\end{proof}

See also \cite[Proposition~2.3]{Sai} for a version of a second fundamental sequence. We note that by \cite[Proposition~2.6]{Sai}, if $(R,\m,k)$ is local, with $k$ of characteristic $p>0$, and $A=\Z$, the sequence
\[ F_{k}(\m/\m^2) \to k\otimes_R \widetilde{\Omega}_{R|(\Z,\delta)} \to \widetilde{\Omega}_{k|(\Z,\delta)} \to 0\]
is split-exact over $R/pR$.

Our next goal is to give a construction of universal perivation modules.

\begin{construction}\label{construction-universal}
	Let $(A,\delta)$ be a mod $p^2$ $\delta$-ring. Let $T$ be an $A$-algebra with
	\begin{enumerate}
		\item $\varepsilon$, a $p$-derivation mod $p^2$ on $T$ extending $\delta$, and
		\item $\{ t_\lambda \} \subset T$ such that $\{ d t_\lambda \, | \, \lambda\in\Lambda\}$ form a free basis for $\Omega_{T/pT|A/pA}$.
	\end{enumerate}
	Let $\{\partial_\lambda \, | \, \lambda\in\Lambda\}$ be the dual basis to $\{ d t_\lambda\}$ in $\Der_{T/pT|A/pA}(T/pT)$.

	We set $\widetilde{\Omega}_{T|(A,\delta)}$ to be the free $T/pT$-module with basis $\{\tilde{d}p\} \cup \{\tilde{d}t_{\lambda} \, | \, \lambda\in \Lambda\}$. Let ${\tilde{d}_{T|(A,\delta)}:T \to \widetilde{\Omega}_{T|(A,\delta)}}$ be the perivation given by
	\[	\tilde{d}_{T|(A,\delta)}(t)=\varepsilon(t) \tilde{d}p + \sum_{\lambda\in \Lambda} (\partial_\lambda(t))^p \tilde{d}t_{\lambda}.\]
	Note that since this is a sum of a $p$-derivation mod $p^2$ and derivations of the Frobenius, it is indeed a perivation.
\end{construction}

\begin{proposition}\label{prop-universal-T}
	In the setting of Construction~\ref{construction-universal}, the map $\tilde{d}_{T|(A,\delta)}$ is a universal perivation, and $\widetilde{\Omega}_{T|(A,\delta)}$ is a universal perivation module.
\end{proposition}
\begin{proof}
	For any $T/pT$-module $M$, and any collection $\{m_p\} \cup \{ m_\lambda \,|\, \lambda\in \Lambda\}$ of elements of $M$, there is a unique $T/pT$-module homomorphism $\varphi:\widetilde{\Omega}_{T|(A,\delta)}\to M$ such that the composition $\varphi \circ \tilde{d}_T$ is a perivation that sends $p$ to $m_p$ and $t_\lambda$ to $m_\lambda$. Thus, it suffices to show that any perivation is uniquely determined by its values on $p$ and $\{t_\lambda\,|\, \lambda\in \Lambda\}$.
	
	To see this, let $\alpha$ be a perivation, and let $\beta=\alpha-\alpha(p)\varepsilon$. Then $\beta$ is a derivation of the Frobenius, so we can write $\beta=\gamma \circ d_{T/pT|A/pA}$, where $d_{T/pT|A/pA}$ is the universal derivation $T/pT\to \Omega_{T/pT|A/pA}$, and $\gamma$ is $p$-linear. As $\gamma$ is determined by its values on $\{ d t_\lambda \, | \, \lambda\in\Lambda\}$, $\beta$ is determined by its values on $\{t_\lambda\,|\, \lambda\in \Lambda\}$, as required.
\end{proof}

\begin{remark}\label{rem:conditions-universal} Construction~\ref{construction-universal} applies in the following settings:
	\begin{enumerate}
		\item Let $T=A[\{x_\lambda \,|\, \lambda\in \Lambda\}]$ be a polynomial ring in an arbitrary set of variables over $A$. Then, for $\wt{\Omega}_{T|(A,\delta)}$, we may take $\{t_\lambda\}$ to be the set of variables, and $\{\partial_\lambda\} = \{ \frac{\partial}{\partial x_\lambda}\}$ for $\wt{\Omega}_{T|(A,\delta)}$.
		\item Let $T=V[\{x_\lambda \,|\, \lambda\in \Lambda\}]$ be a polynomial ring in an arbitrary set of variables over a discrete valuation ring $V$ with uniformizer $p$. Then, for $\wt{\Omega}_{T|(\Z,\delta)}$, we may take $\{t_\lambda\}$ to be $S_1 \cup S_2$, where $S_1$ is the set of variables, and $S_2=\{\gamma_\lambda\}$ is a set of elements of $V$ that maps bijectively to a $p$-base for $V/pV$. Then $\{\partial_\lambda\}=D_1 \cup D_2$, where $D_1$ is the set of derivations of the form $\frac{\partial}{\partial x_\lambda}$, and $D_2$ is the set of derivations $\{\frac{\partial}{\partial \gamma_\lambda}\}$ obtained from the $p$-base. 
		\item Let $T=V[[ x_1,\dots, x_n ]] [y_1,\dots,y_m]$ be a polynomial ring over power series ring over a discrete valuation ring with uniformizer $p$. Then, as in the previous case, for $\wt{\Omega}_{T|(\Z,\delta)}$, we may take $\{t_\lambda\}$ to be the union of the variables (both power series and polynomial) and a lift of a $p$-base, and $\{\partial_\lambda\}$ to be the union of the derivations with respect to the variables and the derivations induced by the $p$-base.
	\end{enumerate}
\end{remark}

\begin{construction}\label{construction-general}
	Let $(A,\delta)$ be a $\delta$-ring, and $R$ be an $A$-algebra. Let $T$ be an $A$-algebra with $R=T/I$ for some ideal $I$, and
	\begin{enumerate}
		\item $\varepsilon$, a $p$-derivation mod $p^2$ on $T$ extending $\delta$, and
		\item $\{ t_\lambda \} \subset T$ such that $\{ d t_\lambda \, | \, \lambda\in\Lambda\}$ form a free basis for $\Omega_{T/pT|A/pA}$;
	\end{enumerate}
	such a $T$ always exists by Remark~\ref{rem:conditions-universal}. Let $\{\partial_\lambda \, | \, \lambda\in\Lambda\}$ be the dual basis to $\{ d t_\lambda\}$ in $\Der_{T/pT|A/pA}(T/pT)$.

	We set $\widetilde{\Omega}_{R|(A,\delta)}$ to be the quotient of the free $R/pR$-module $P_R$ with basis ${\{\tilde{d}p\} \cup \{\tilde{d}t_{\lambda} \, | \, \lambda\in \Lambda\}}$ by the
	submodule $N$ of elements of the form \[\{\delta(a) \tilde{d}p + \sum_{\lambda\in \Lambda} (\partial_\lambda(a))^p \tilde{d}t_{\lambda} \ | \ a\in I\};\]
	note that $N$ is generated by the set of elements of this form as $a$ varies though any generating set of $I$.
	
	Let $\tilde{d}_{R|(A,\delta)}:R \to \widetilde{\Omega}_{R|(A,\delta)}$ be the perivation given by
	\[	\tilde{d}_{R|(A,\delta)}(r)=\varepsilon(r) \tilde{d}p + \sum_{\lambda\in \Lambda} (\partial_\lambda(r))^p \tilde{d}t_{\lambda}.\]
\end{construction}

\begin{proposition}\label{prop:pres-uni-peri}
	In the context of Construction~\ref{construction-general}, the map $\tilde{d}_{R|(A,\delta)}$ is a universal perivation, and $\widetilde{\Omega}_{R|(A,\delta)}$ is a universal perivation module.
\end{proposition}
\begin{proof} Write $\pi:T\to R$ for the quotient map. Let $M$ be an $R/pR$-module, and $\alpha:R\to M$ be a perivation. Then $\alpha \circ \pi$ is a perivation from $T$ to $M$, so there is a unique $T$-linear $\varphi: \widetilde{\Omega}_{T|(A,\delta)} \to M$ such that $\varphi \circ  \tilde{d}_{T|(A,\delta)} = \alpha \circ \pi$. Since, by construction, $\widetilde{\Omega}_{R|(A,\delta)}$ is a quotient of $\widetilde{\Omega}_{T|(A,\delta)}$ and the map $\tilde{d}_{R|(A,\delta)}$  is consistent with $\tilde{d}_{T|(A,\delta)}$ via this quotient, it suffices to show that $\varphi$ factors through $\widetilde{\Omega}_{R|(A,\delta)}$. 

Since $M$ is an $R/pR$-module, $\varphi$ factors through $R\otimes_T \widetilde{\Omega}_{T|(A,\delta)}$. Note that, by construction, $ \widetilde{\Omega}_{R|(A,\delta)}$ is the quotient of $R\otimes_T \widetilde{\Omega}_{T|(A,\delta)}$ by the image of $I$ under $\tilde{d}_{T|(A,\delta)}$. But then, $0 = \alpha \pi (I) = \varphi \tilde{d}_{T|(A,\delta)}(I)$, so $\varphi$ indeed factors through this quotient.
	\end{proof}
	
	We note that in \cite{DKRZ} and \cite{Sai}, the corresponding universal modules are constructed differently, though our construction agrees with those in the case $(A,\delta) = (\Z,\delta)$, since they satisfy the same universal property.

We note also (cf. \cite[Proposition~2.5]{Sai}):
	\begin{lemma}
		Let $(A,\delta)$ be a mod $p^2$ $\delta$-ring.
		\begin{enumerate}
			\item If $R=\varinjlim_\lambda R_\lambda$ is a direct limit of $A$-algebras, then $\widetilde{\Omega}_{R|(A,\delta)} = \varinjlim_\lambda \widetilde{\Omega}_{R_\lambda|(A,\delta)}$.
			\item If $W\subseteq R$ is multiplicatively closed, then $\widetilde{\Omega}_{W^{-1}R|(A,\delta)} = W^{-1}\widetilde{\Omega}_{R|(A,\delta)}$.
		\end{enumerate}
	\end{lemma}
	\begin{proof}
		\begin{enumerate}
			\item One can construct a universal perivation from $R$ to $\varinjlim_\lambda \widetilde{\Omega}_{R_\lambda|(A,\delta)}$ using the universal properties. Namely, since the direct limits $\varinjlim_\lambda R_\lambda$ and $\varinjlim_\lambda \widetilde{\Omega}_{R_\lambda|(A,\delta)}$ can be computed at the level of sets, there is a canonical function $\varinjlim_\lambda R_\lambda \to \varinjlim_\lambda \widetilde{\Omega}_{R_\lambda|(A,\delta)}$; since any two elements  $x,y\in R$ are in the image of some $R_\lambda$, and the canonical map from $R_\lambda \to \varinjlim_\lambda \widetilde{\Omega}_{R_\lambda|A}$ satisfies the conditions of Definition~\ref{def:peri}, the canonical map $\varinjlim_\lambda R_\lambda \to \varinjlim_\lambda \widetilde{\Omega}_{R_\lambda|(A,\delta)}$ does as well. Now, given a perivation from $R$ to $M$, we obtain a consistent system of perivations from $R_\lambda$ to $M$ which induces a unique consistent system of $R_\lambda$-linear homomorphisms from $\widetilde{\Omega}_{R_\lambda|(A,\delta)}$ to $M$. This induces a unique factorization through $\varinjlim_\lambda \widetilde{\Omega}_{R_\lambda|(A,\delta)}$.
			\item Since we can write $W^{-1}R=\varinjlim_{f\in W} R_f$, it suffices to check for principal localizations. Moreover, using the second fundamental sequence, it suffices to check for the case of a polynomial ring. In this case, we can apply Construction~\ref{construction-general} to the polynomial ring with one extra variable $t$ modulo a relation of the form $ft-1$ and verify directly.
			\qedhere
		\end{enumerate}
	\end{proof}

\begin{remark}
Since formation of $\widetilde{\Omega}$ commutes with localization, Construction~\ref{construction-universal} is valid for localizations of polynomial rings over power series rings over unramified discrete valuation rings.
\end{remark}

	\section{Jacobian criterion}\label{sec:Jac.Crit.fg}           

\subsection{Theorem~A} In this section, we prove Theorem~A from the introduction. We also give a self-contained proof of the analogous criterion for the singular locus among primes that do not contain $p$, which is originally a result of Seydi \cite{Sey}.

	\begin{definition}\label{def-mixed-J}
		Let $(V,pV,K)$ be an unramified discrete valuation ring of mixed characteristic $(0,p)$. Let $T$ be a polynomial ring  over a power series ring over $V$ with $n$ total variables $\vect xn$ of both types combined, and let $\delta$ be a $p$-derivation mod $p^2$ on $T$; cf. Corollary~\ref{cor:p-der-anythingoveranyV}. In particular, we allow $T$ to be a polynomial ring or a power series ring. Fix a sequence of elements $\{ \gamma_\lambda\}$ of $V$ that map bijectively to a $p$-base of $K$.

		For a sequence of elements $\bsf=(f_1,\dots,f_a)$ in $R$, the
		\emph{mixed Jacobian matrix} of $\bsf$ is
	the matrix $\widetilde{J}(\bsf)$ over $T/pT$ with
	\begin{itemize}
		\item rows indexed by $\{f_1,\dots,f_a\}$,
		\item columns indexed by $\{p\} \cup \{x_1,\dots,x_n\} \cup \{ \gamma_\lambda\}$, and
		\item the entry in the row indexed by $f_i$ and column indexed by $g$ is  $\delta(f_i)$ if $g=p$, and
	$( \frac{\partial f_i}{\partial g})^p$ if $g\neq p$.
	\end{itemize} 
The maps $\partial/\partial \gamma_\lambda$ are the derivations discussed in Remark~\ref{rem:conditions-universal}.
	\end{definition}
	
	We note that the terminology ``mixed Jacobian matrix'' is used in the sources \cite{Nag,Sam,Zar} to denote Jacobian matrices in equal characteristic that include derivations of inseparability akin to those discussed in Remark~\ref{rem:conditions-universal}. As the term used in these sources does not seem to be in common usage currently, we use it here for this new notion in mixed characteristic.

\begin{remark}
	If $(V,pV,K)$ is a discrete valuation ring with $K$ perfect, then \[ \widetilde{J}(\bsf) = \begin{bmatrix} \delta(f_1) & (\frac{\partial f_1}{\partial x_1})^p & \cdots & (\frac{\partial f_1}{\partial x_n})^p  \\
	\vdots & \vdots & \ddots & \vdots  \\
	\delta(f_a) & (\frac{\partial f_a}{\partial x_1})^p & \cdots & (\frac{\partial f_a}{\partial x_n})^p\end{bmatrix}.  \]
\end{remark}

\begin{remark} The mixed Jacobian matrix depends on the choice of $\delta$ and the choice of a $p$-base for $K$, although these are suppressed from the notation. However, the cokernel of the matrix $\widetilde{J}(\bsf)$ considered as a matrix over $T/(p,\bsf)T$, is isomorphic to the module $\widetilde{\Omega}_{T|(\Z,\delta)}$ by Construction~\ref{construction-general} and Remark~\ref{rem:conditions-universal}, and thus is independent of these choices. Cf. \cite[Footnote~4]{Nag} for this observation in the equal characteristic case.
	\end{remark}

\begin{definition}	In the same setting as Definition~\ref{def-mixed-J}, we write ${J}(\bsf)$ for the classical Jacobian matrix, with
		\begin{itemize}
		\item rows indexed by $\{f_1,\dots,f_a\}$,
		\item columns indexed by $\{x_1,\dots,x_n\}$, and
		\item the entry in the row indexed by $f_i$ and column indexed by $g$ is 
		$\frac{\partial f_i}{\partial g}$.
	\end{itemize} 
\end{definition}

\begin{discussion}\label{disc-kunz}
	We briefly discuss the notion of module of K\"ahler differentials for quotients of polynomial rings over power series rings that we will need for primes not containing $p$ in Theorem~\ref{thm-A}. We refer the reader to the book of Kunz \cite{Kun} for more details on the notions discussed here; see also \cite{Jia} for a concise treatment of these notions.
	
	Let $(V,pV,K)$ be a complete discrete valuation ring, $S$ a power series ring over $V$, and $T$ a polynomial ring over $S$, and $\p$ a prime ideal of $T$. Then, there is a \textit{universally finite derivation} $\partial: S \to \Omega'_{S|V}$: a derivation such that for any finitely generated $S$-module $M$ and derivation $\zeta$, there is a unique $S$-module homomorphism $\phi$ such that $\zeta = \phi \circ \partial$. There is a \textit{universal extension} $\hat{\partial}: T_\p \to \Omega_{T_\p|\partial}$ of $\partial$: for every $T$-module $M$ and $S$-module homomorphism $\omega: \Omega'_{S|V} \to M$ and derivation $\sigma:T\to M$ such  that the restriction of $\sigma$ to $R$ is $\omega \circ \partial$, $\sigma$ factors through $\hat{\partial}$ compatibly with the other maps.
	
	By \cite[Corollary~4.22]{Kun}, we can identify $\Omega_{T_\p|\partial}$ with $(\Omega_{T|\partial})_\p$, and by \cite[Formula~4.11(b)~and~Example~12.7]{Kun}, we can realize $\Omega_{T|\partial}$ as a free module with basis given by $dx_i$ for each variable $x_i$, and $\hat{\partial}:T \to \Omega_{T|\partial}$ as $\hat{\partial}(f) = \sum_i \frac{\partial f}{\partial i} dx_i$.
\end{discussion}


\begin{lemma}\label{lem:rankimage0} Let $(V,pV,K)$ be a complete discrete valuation ring with uniformizer $p\in \Z$, and $T$ be a polynomial ring  over a power series ring over $V$ in $n$ variables $\vect xn$ total. Let $I$ be an ideal of $R$, with generating set $\bsf=(f_1,\dots,f_a)$, and let $\mathfrak{p}$ be a prime of $R$ containing~$I$, but not containing $p$.
	Then, the $\kappa(\p)$-vector space dimension of the image of $I T_\p$ in $\p T_\p/\p^2 T_\p$ is equal to the rank of the Jacobian matrix ${J}(\bsf)$ considered as a matrix in~$\kappa(\p)$.
\end{lemma}	
\begin{proof}
In the notation of Discussion~\ref{disc-kunz} by \cite[Proposition~13.14]{Kun}, the map \[\p T_\p/\p^2 T_\p \to \Omega_{T_\p|\partial} \otimes_{T_\p} \kappa(\p)\] induced by $\hat{\partial}$ is injective. Thus,
			\[\dim_{\kappa(\p)} \big(\mathrm{im}(I T_\p \to \p T_\p/\p^2 T_\p)\big) =  \dim_{\kappa(\p)} \big(\mathrm{im}(I T_\p \xrightarrow{\hat{\partial}}  {\Omega}_{T_{\p}|\partial} \otimes_T \kappa(\p))\big).\]
			By Discussion~\ref{disc-kunz}, this image is the $\kappa(\p)$-subspace of $\Omega_{T|\partial} \otimes_{T} \kappa(\p)$ generated by the rows of the classical Jacobian matrix (with free basis for  $\Omega_{T|\partial} \otimes_{T} \kappa(\p)$ as in Discussion~\ref{disc-kunz}).
\end{proof}

Applying the universal perivation modules, we obtain a similar result for primes that contain $p$.

\begin{lemma}\label{lem:rankimagep} Let $(V,pV,K)$ be a discrete valuation ring with uniformizer $p\in \Z$, and $T$ be a polynomial ring  over a power series ring over $V$ in $n$ variables $\vect xn$ total. Let $I$ be an ideal of $R$, with generating set $\bsf=(f_1,\dots,f_a)$, and let $\mathfrak{p}$ be a prime of $R$ containing~$I$ and $p$. Let $\delta$ be a $p$-derivation modulo $p^2$ on $T$; such a $\delta$ exists by Corollary~\ref{cor:p-der-anythingoveranyV}.
Then, the $\kappa(\p)$-vector space dimension of the image of $I T_{\p}$ in $\p T_{\p}/\p^2 T_{\p}$ is equal to the rank of the mixed Jacobian matrix $\widetilde{J}(\bsf)$ considered as a matrix in~$\kappa(\p)$.
\end{lemma}	
\begin{proof}
By \cite[Proposition~2.6]{Sai}, the map $F_{\kappa(\p)}(\p T_{\p}/\p^2 T_{\p})  \to \widetilde{\Omega}_{T_\p|(\Z,\delta)} \otimes_T \kappa(\p)$ induced by the universal perivation on $T_\p$ is injective. 

		\begin{align*} \dim_{\kappa(\p)}& \big(\mathrm{im}(I T_\p \to \p T_\p /\p^2 T_\p )\big) \\&=  \dim_{\kappa(\p)} \big(\mathrm{im}( I T_\p/ (\p^2 T_\p \cap  I T_\p) \to \p T_\p/\p^2  T_\p)\big) \\ &=  \dim_{\kappa(\p)} \big(\mathrm{im}(F_{\kappa(\p)}(I T_\p/ (\p^2  T_\p \cap I T_\p)) \to F_{\kappa(\p)}(\p T_\p/\p^2 T_\p))\big)  
		\\&= \dim_{\kappa(\p)} \big(\mathrm{im}(F_{\kappa(\p)}(I T_\p/ (\p^2  T_\p \cap I T_\p)) \to  \widetilde{\Omega}_{T_\p|\Z} \otimes_{T_\p} \kappa(\p))\big),
		\end{align*}
		where the map on the second line is induced from the previous map  by applying the Frobenius functor, and the last map is obtained from the previous map by postcomposing with the map $\rho$. It remains only to note that, by Construction~\ref{construction-universal}, the image of this map is the ${\kappa(\p)}$-subspace of $\widetilde{\Omega}_{T_\p|\Z} \otimes_{T_\p} \kappa(\p)$ generated by the rows of the mixed Jacobian matrix (given a compatible choice of free basis for $\Omega_{T/pT|\Z/p\Z}$ in Construction~\ref{construction-universal} and $p$-base in  Definition~\ref{def-mixed-J}).
\end{proof}

We will also need the following fact.

\begin{proposition}\label{prop:greg} Let $V \to W$ be a local map of discrete valuation rings of mixed characteristic $p$, and let $B$ denote a local $V$-algebra such that
	$B$ has maximal ideal $pB +\fb$ and is complete with respect to $\fb$. Let $C = W \wh{\otimes}^{\fb}_V B$ be the completion of $W \otimes_V B$ with respect to the image of $\fb$. Then $C[1/p]$ is geometrically regular over $B[1/p]$.
	In particular, if $\vect x d$ are indeterminates, $V[[\vect x d]][1/p] \to W[[\vect x d]][1/p]$ is geometrically regular. \end{proposition}
\begin{proof}  It suffices to prove geometrically regularity when $B$ has the special form $A =V[[\vect x d]]$: one may then get the  general case by a base change
	from a suitable choice of $A$ to $B$.  Flatness is obvious in the case of $A$,  even if we do not localize at $p$.  We are in equal characteristic 0, and so
	it suffices to show that the fiber of $B \to C$ over a prime $\q$ not containing $p$ is regular.  Thus, we may replace $B$ by $B/\q$,  and we only need to 
	consider the generic fiber.  We change notation and use $B$ for the base domain.  We may extend $p$ to a system of parameters $\vect u h$ for
	$\m_B$ where the $u_j \in \fb$.  Then we may map a new choice of $A$, namely, $V[[\vect xh]]$, to $B$ as an $V$-algebra so that
	$\vect x h$ map to $\vect u h$, using the completeness of $B$ with respect to $\fb$.  Then $B$ is module-finite over this choice of $A$,
	and the problem of proving geometric regularity reduces to checking regularity of the generic fiber for this choice of $A$, i.e., for
	$V[[\vect x h]] \to W[[\vect x h]]$.  This is obvious since the target ring is now regular. \end{proof}

Part (1) of the following Theorem is originally due to Nagata \cite{Nag2} in the affine case and to Seydi \cite{Sey} in the generality below.

\begin{theorem}\label{thm-A} Let $(V,pV,k)$ be a discrete valuation ring  with uniformizer $p$ and $T$ be a polynomial ring  over a power series ring over $V$ in $n$ variables $\vect xn$ total. Let $I=(f_1,\dots,f_a)T$ be an ideal of pure height~$h$ in~$T$, and~$R=T/I$.
	Then, for a prime $\p$ of $R$,	
	\begin{enumerate}
		\item If $p\notin \p$, then $R_{\p}$ is regular if and only if the image of the ideal $I_h({J}(\bsf))$ of $h\times h$ minors of the classical Jacobian matrix ${J}(\bsf)$ is not contained in~$\p$;
		\item If $p\in \p$, then $R_{\p}$ is regular if and only if the image of the ideal $I_h(\widetilde{J}_{T}(\bsf))$ of $h\times h$ minors of the mixed Jacobian matrix $\widetilde{J}(\bsf)$ is not contained in $\p$.
	\end{enumerate}
\end{theorem}
\begin{proof}
	In the first case, we note that by Proposition~\ref{prop:greg}, we can reduce to the case where $V$ is complete. Indeed, the map from $R$ to $\wh{V} \wh{\otimes}^{\fb}_V R$, where $\fb$ is the ideal generated by the power series variables, is geometrically regular after inverting $p$, so the singular locus of the source expands to the singular locus of the target; the classical Jacobian matrix is also the same for both rings.
	
	Then, we compute
	\[ \mu = \dim_{\kappa(\p)}(\p R_\p / \p^2 R_{\p}) - \dim(R_\p). \]
	We observe that 	$ \dim(R_\p)= \mathrm{ht}(\p) - h$ and
	\[\dim_{\kappa(\p)}(\p R_\p / \p^2 R_{\p})= \dim_{\kappa(\p)}(\p/ \p^2) -  \dim_{\kappa(\p)}(\mathrm{im}(I\to \p/ \p^2)).\]
	Using Lemma~\ref{lem:rankimage0} and regularity of $R_{\p}$, we have that 
	$ \dim_{\kappa(\p)}(\p R_\p / \p^2 R_{\p}) $ is equal to $\dim(R_\p)$ minus the rank of the classical Jacobian matrix ${J}(\bsf)$ considered as a matrix over $\kappa(\p)$. Thus,
	\[ \mu = h - \mathrm{rank} ({J}(\bsf)_{\kappa(\p )}).\]
	Thus, $R_{\p}$ is regular if and only if $\mu=0$, which happens if and only if the rank of ${J}(\bsf)$ considered as a matrix over $\kappa(\p)$ is equal to $h$, and this in turn happens if and only if $\p$ does not contain the $h\times h$ minors of ${J}(\bsf)$.
	
In the latter case, we proceed similarly. 
In this case, we use Lemma~\ref{lem:rankimagep} to conclude that
	$ \dim_{\kappa(\p)}(\p R_\p / \p^2 R_{\p}) $ is equal to $\dim(R_\p)$ minus the rank of the mixed Jacobian matrix $\widetilde{J}(\bsf)$ considered as a matrix over $\kappa(\p)$. Thus,
	\[ \mu = h - \mathrm{rank} (\widetilde{J}(\bsf)_{\kappa(\p )}),\]
	and $R_{\p}$ is regular if and only if the rank of $\widetilde{J}(\bsf)$ considered as a matrix over $\kappa(\p)$ is equal to $h$, which in turn happens if and only if $\p$ does not contain the $h\times h$ minors of ${J}(\bsf)$.
\end{proof}


	\begin{example}
		Let $R=\displaystyle\frac{\Z[x]}{x^2-n} \cong \Z[\sqrt{n}]$ for $n\in \Z$ a squarefree nonunit. Set $f=x^2-n$. For $p$ prime, we take $\widetilde{J}(f)=\begin{bmatrix} \delta_p(f) & (2x)^p \end{bmatrix}$, where $\delta_p$ is the unique $p$-derivation on $\Z$ extended to $\Z[x]$ via $\delta_p(x)=0$. Observe that	\[\delta_p(x^2-n)=\delta_p(-n)-C_p(x^2,n)\in \delta_p(-n) + (x).\]
		
		For $p$ odd, \begin{align*}
		{\V_R(p,I_1(\widetilde{J}(f)))}&\cong{\V_{\Z[x]}(p,\delta_p(f),(2x)^p,x^2-n)}={\V_{\Z[x]}(p,\delta_p(f),x,x^2-n)}\\
		&={\V_{\Z[x]}(p,\delta_p(f),x,n)}\cong{\V_{\Z}(p,n,\delta_p(-n))}.
		\end{align*} If $p \nmid n$, this is empty since $(p,n)=\Z$. If $p\,|\,n$, since $p^2 \nmid n$, we have $p \nmid \delta_p(-n)$, so this locus is empty again.
		
		For $p=2$,  \begin{align*}\V_R(p,I_1(\widetilde{J}(f)))&\cong\V_{\Z[x]}{(2,\delta_2(f),(2x)^p,x^2-n)}\\&=\V_{\Z[x]}{(2,\delta_2(f),x^2-n)}=\V_{\Z[x]}{(2,\delta_2(-n)+nx^2,x^2-n)}.\end{align*}
		If $n\equiv 1$ mod $4$, then $\delta_2(-n)=-\frac{n^2+n}{2}$ is odd, so \[{\V_{\Z[x]}(2,\delta_2(-n)+nx^2,x^2-n)}={\V_{\Z[x]}(2,x^2-1)}\] is nonempty, corresponding to the singular maximal ideal $(2,\sqrt{n}-1) \subset R$.

			If $n\equiv 2$ mod $4$, then $\delta_2(-n)$ is odd, so \[\V_{\Z[x]}(2,\delta_2(-n)+nx^2,x^2-n)=\V_{\Z[x]}(2,1,x^2)\] is empty.
				If $n\equiv 3$ mod $4$, then $\delta_2(-n)$ is even, so \[\V_{\Z[x]}(2,\delta_2(-n)+nx^2,x^2-n)=\V_{\Z[x]}(2,x^2,x^2-1)\] is empty.				
				Of course, this agrees with the basic number theory fact that for a squarefree nonunit $n$, $\Z[\sqrt{n}]$ is integrally closed in its fraction field (and equivalently regular, as $R$ is a one-dimensional domain) if and only if $n \not\equiv 1 \ \mathrm{mod}\  4$.
	\end{example}

\begin{remark} Theorem~\ref{thm-A} can be used to give a quick proof that for any finitely generated algebra over any complete local ring of mixed characteristic, the singular locus is closed, given the result in equal characteristic. We can write such a ring as a quotient of a polynomial ring over a power series ring over an unramified discrete valuation ring $(V,pV)$, say $T$. Given a primary decomposition $Q_1 \cap \cdots \cap Q_n$ of the presenting ideal $I$, the localization at any prime $\p$ that contains two $Q_i$'s or some $Q_i$ that is not a minimal prime is not a domain, and hence not regular: this is a closed set, and the remainder of the singular locus is the singular locus of the quotients by the minimal primes, so we can reduce to the case that the ring $T/Q$ is a domain. If $p\in Q$, we reduce to the equal characteristic case. Otherwise, the singular locus of $R=T/Q$ is described by Theorem~\ref{thm-A}. A priori, this describes the singular locus as the union of a closed set and a locally closed set, but Nagata's openness criterion \cite[Theorem~24.2]{Mat} applies readily to the complement.\end{remark}


\subsection{Elementary proof}\label{subsec:elem}

We now provide a concrete self-contained proof of Theorem~A in the affine case. For simplicity, we assume that the residue field of the base discrete valuation ring is algebraically closed, though the perfect case follows easily. 

	\begin{theorem}
	Let $(V,pV,K)$ be an unramified discrete valuation ring of mixed characteristic. Assume that $K$ is algebraically closed. Let $T=V[x_1,\dots,x_n]$, and~$\delta$ be a  $p$-derivation mod $p^2$ on~$T$ extending the $p$-derivation mod $p^2$ on $V$. Let $I=(f_1,\dots,f_a)$ be an ideal of pure height~$h$ in~$T$, and~$R=T/I$. Let $\p$ be a prime ideal of $R$ with $p\in \p$.
	
Then, the mixed Jacobian matrix $\widetilde{J}(\bsf)$ considered in $\kappa(\p)$ has rank at most~$h$, and equality holds if and only if $R_{\p}$ is regular.
	\end{theorem}
\begin{proof}
First, we consider the case where $\p$ is a maximal ideal $\m$. Let $\n$ be the preimage of $\m$ in $T$. It follows from the Nullstellensatz that the image of $\n$ in $T/pT$ is of the form $(x_1-a_1,\dots,x_{n}-a_{n})$ for some $a_1,\dots,a_{n}\in K$, so $\n$
	is of the form $\n=(p,x_1-v_1,\dots,x_{n}-v_{n})$ for some $v_1,\dots,v_{n}\in V$ with $v_i= a_i+ pV$, and $f_j(v_1,\dots,v_n)\in pV$ for $j=1,\dots,a$. 
	
	We aim to compute the dimension of $\m/\m^2$ as a $K$-vector space. To this end,  there is a short exact sequence
	\[ 0 \to \frac{I}{\n^2\cap I} \to \frac{\n}{\n^2} \to \frac{\m}{\m^2} \to 0, \]
	so $R_{\m}$ is regular if and only if the $K$-vector space dimension of ${I}/({\n^2}\cap I)$ is equal to ${\dim(T_{\n})-\dim(R_{\m})}$. Since $T_{\n}$ is catenary, and each minimal prime of $I$ has the same height, we have that $\dim(T_{\n})-\dim(R_{\m})=\mathrm{ht}(I T_\n)=h$ for $\n \supseteq I$. 
	
	Write $\tilde{x}_i=x_i-v_i$ for $i=1,\dots,n$, so that $\n=(p,\tilde{x}_1,\dots,\tilde{x}_{n})$. Using Taylor's formula, for each $f_j$, we get 
	\begin{align*} f_j &= \sum_{(\alpha_1,\dots,\alpha_{n})} \frac{\partial^{\alpha_1+\dots+\alpha_n} \, f_j}{\partial x_1^{\alpha_1} \cdots \partial x_{n}^{\alpha_{n}}}(v_1,\dots,v_{n}) \tilde{x}_1^{\alpha_1} \cdots \tilde{x}_{n}^{\alpha_{n}}\\
	&\equiv f_j(v_1,\dots,v_{n}) + \sum_{i=1}^{n} \frac{\partial f_j}{\partial x_i}(v_1,\dots,v_{n}) \tilde{x}_i  \ \mathrm{mod} \ \n^2. 
	\end{align*}
	
By Remark~\ref{remark:der.prod}(\ref{item:der.prod6}) we then have
\[ \delta(f_j) \equiv \delta\left(f_j(v_1,\dots,v_n) + \sum_{i=1}^{n} \frac{\partial f_j}{\partial x_i}(v_1,\dots,v_{n}) \tilde{x}_i \right) \ \mathrm{mod} \ \n.\]
 
 Since $f_j(v_1,\dots,v_n)\in pV \subseteq \n$, we have 
\[C_p\left(f_j(v_1,\dots,v_n),\left(\sum_{i=1}^{n} \frac{\partial f_j}{\partial x_i}(v_1,\dots,v_{n}) \tilde{x}_i\right)\right) \in \n,\]
 so 
 \[ \delta(f_j) \equiv \delta\left(f_j(v_1,\dots,v_n) \right)+ \delta \left(\sum_{i=1}^{n} \frac{\partial f_j}{\partial x_i}(v_1,\dots,v_{n}) \tilde{x}_i \right) \ \mathrm{mod} \ \n.\]
 
 Likewise, since $\tilde{x}_i\in \n$, $C_p(\{\frac{\partial f_j}{\partial x_i}(v_1,\dots,v_{n}) \tilde{x}_i\})\in \n$, and hence
 \[ \delta(f_j) \equiv \delta(f_j(v_1,\dots,v_n)) + \sum_{i=1}^{n} \delta\left(\frac{\partial f_j}{\partial x_i}(v_1,\dots,v_{n}) \tilde{x}_i \right) \ \mathrm{mod} \ \n.\]
 By the product rule for $p$-derivations, we then have
	\[ \delta(f_j) \equiv \delta(f_j(v_1,\dots,v_n)) + \sum_{i=1}^{n} \left(\frac{\partial f_j}{\partial x_i}(v_1,\dots,v_{n})\right)^p \delta(\tilde{x}_i)\ \mathrm{mod} \ \n.\]
	
Note that, for $\alpha\in V$, by Remark~\ref{remark:der.prod}(\ref{item:der.prod7}),
 if we write $f_j(v_1,\dots,v_{n}) \equiv p \alpha \ \mathrm{mod} \ p^2 V$, we obtain that $\delta(f_j(v_1,\dots,v_{n})) \equiv \alpha^p$ in $K=V/pV$.
 Thus, 
 \[ \alpha \equiv {\delta(f_j)}^{1/p} - \sum_i \frac{\partial f_j}{\partial x_i}(v_1,\dots,v_n) {\delta(\tilde{x_i})}^{1/p} \ \mathrm{mod} \ \n.\]
Using this equivalence to substitute in for $f_j(v_1,\dots,v_n)$ in the first expression for $f_j$ modulo $\n^2$ above, we get that the $K$-linear expression for $[f_j] \in \n/\n^2$ in terms of the basis
  
  \[ [p] , [\tilde{x_1} - p {\delta(\tilde{x_1})}^{1/p}], \dots, [\tilde{x_n} - p {\delta(\tilde{x_n})}^{1/p}] \] is
	\[ [f_j] = \overline{\delta(f_j)(\underline{v})}^{1/p} [p] + \overline{\frac{\partial f_j}{\partial{x_1}}(\underline{v})}[\tilde{x}_1- p {\delta(\tilde{x_1})}^{1/p}] + \cdots + \overline{\frac{\partial f_j}{\partial{x_{n}}}(\underline{v})}[\tilde{x}_{n}- p {\delta(\tilde{x_n})}^{1/p}],  \]
	where bars denote images modulo $\n$.
	Thus, the dimension of $I/(\n^2\cap I)$ as a $K$-vector space is the rank of the matrix
	\[\renewcommand\arraystretch{2} \begin{bmatrix}  \overline{\delta(f_1)(\underline{v})}^{1/p} & \overline{\frac{\partial f_1}{\partial x_1}(\underline{v})} &  \dots & \overline{\frac{\partial f_1}{\partial x_{n}}(\underline{v})} \\ \overline{\delta(f_2)(\underline{v})}^{1/p} &
	\overline{\frac{\partial f_2}{\partial x_1}(\underline{v})} &  \dots & \overline{\frac{\partial f_2}{\partial x_{n}}(\underline{v})} \\  
	\vdots & \vdots & \ddots & \vdots \\ \overline{\delta(f_a)(\underline{v})}^{1/p} &
	\overline{\frac{\partial f_{a}}{\partial x_1}(\underline{v})} &  \dots & \overline{\frac{\partial f_a}{\partial x_{n}}(\underline{v})} \end{bmatrix}. \]
	Raising elements to $p$th powers in $K$ does not affect which minors are zero, so the rank of the matrix above is the same as that of
	\[\renewcommand\arraystretch{2} \begin{bmatrix}  \overline{\delta(f_1)(\underline{v})} & \overline{\frac{\partial f_1}{\partial x_1}(\underline{v})}^p &  \dots & \overline{\frac{\partial f_1}{\partial x_{n}}(\underline{v})}^p \\ \overline{\delta(f_2)(\underline{v})} &
	\overline{\frac{\partial f_2}{\partial x_1}(\underline{v})}^p &  \dots & \overline{\frac{\partial f_2}{\partial x_{n}}(\underline{v})}^p \\  
	\vdots & \vdots & \ddots & \vdots \\ \overline{\delta(f_a)(\underline{v})} &
	\overline{\frac{\partial f_{a}}{\partial x_1}(\underline{v})}^p &  \dots & \overline{\frac{\partial f_a}{\partial x_{n}}(\underline{v})}^p \end{bmatrix}. \]
	This is just the image of $\widetilde{J}(\bsf)$ in $T/\n=R/\m$, and thus the $K$-dimension of $I/(\n^2 \cap I)$ is equal to the rank of $\widetilde{J}(\bsf)$ considered in $\kappa(\m)$.
	
	Since the $K$-vector space dimension of $I/(\n^2 \cap I)$ is less than or equal to the height with equality if and only if $R_{\m}$ is regular, the theorem is established in the case of a maximal ideal.
	
	Now, the inequality on the rank of $\widetilde{J}(\bsf)$ for general prime ideals follows by a basic semicontinuity argument, since maximal ideals are dense in $V_R(p)$. Likewise, the characterization of equality in terms of nonsingularity follows, since $R$ is excellent, and so the singular locus is closed.
\end{proof}

\subsection{Theorem B}

We now give the proof of Theorem~B.

	\begin{theorem}\label{thm-B}
		Let $V$ be a discrete valuation ring with uniformizer $p$, and F-finite residue field. Let $(R,\m,k)$ a local ring with $p\in \m$ that is essentially of finite type over a complete $V$-algebra. 
		
	The local ring $R$ is regular if and only if $\widetilde{\Omega}_{R|(\Z,\delta)}$ is free of rank ${\dim(R)+\log_p[k:k^p]}$.
	\end{theorem}
	\begin{proof} Let $K=V/pV$, and set $a=\log_p[K:K^p]$ and $b=\log_p[k:k^p]$. Write $R\cong (T/I)_\p$, where $T$ is a polynomial ring over a power series ring over $V$, and $\p$ is a prime ideal of $T$ that contains $I$. We observe that $\widetilde{J}(\bsf)$ is a presentation matrix for $\widetilde{\Omega}_{R|(\Z,\delta)}$, with $\dim(T)+a$ generators. 
		
		First, we claim that $\dim(T)+a=\height(I_\p) + \dim(R) + b$. We use the formula $\log_p[\kappa(\q):\kappa(\q)^p]-\log_p[\kappa(\q'):\kappa(\q')^p]=\height(\q'/\q)$ for $\q\subseteq \q'$ in an F-finite ring of positive characteristic; cf., \cite{Kun1}, \cite[Lemma~42.7]{Mat}. Applying this in $T$ with the maximal ideal generated by the variables and the zero ideal, and then with $\p$ and the zero ideal, we have that \[\dim(T)+a=\log_p[\fra(T):\fra(T)^p]=\height(\p) + b.  \]
		Since $\dim(R)=\height(\p)-\height(I_\p)$, the claim holds.
		
		Now, by the characterization of locally free modules via Fitting ideals (cf.~\cite[Proposition~D.8]{Kun}), $\widetilde{\Omega}_{R|(\Z,\delta)}$ is locally free of rank $\dim(R) +b$ if and only if 
		\[ F_{\dim(R) +b}(\widetilde{\Omega}_{R|(\Z,\delta)}) = R \qquad \text{and} \qquad F_{<\dim(R) +b}(\widetilde{\Omega}_{R|(\Z,\delta)}) = 0. \]
		Using the equality above and the presentation by the mixed Jacobian matrix, $\widetilde{\Omega}_{R|(\Z,\delta)}$ is locally free of rank $\dim(R) +b$ if and only if
		\[ I_{\height(I_\p)}(\widetilde{J}(\bsf)) = R \qquad \text{and} \qquad I_{>\height(I_\p)}(\widetilde{J}(\bsf)) = 0. \]
		Now, if $I_{\height(I_\p)}(\widetilde{J}(\bsf)) = R$, then
		by Lemma~4.6, $\dim_k(\mathrm{im}(I_\p \to\p/\p^2))=\height(I_\p)$, so $I_\p$ can be generated by $\height(I_\p)$ elements; in this case we can replace $\bsf$ by such a generating set and the latter condition above is automatic. Thus, $\widetilde{\Omega}_{R|\Z}$ is locally free of rank $\dim(R) +b$ if and only if $I_{\height(I_\p)}(\widetilde{J}(\bsf)) = R$, which happens if and only if $I_{\height(I_\p)}(\widetilde{J}(\bsf)_k) = k$ which, by Lemma~4.6 again, is equivalent to $\dim_k(\mathrm{image}(I_\p \to \p/\p^2))=\height(I_\p)$. But this just means that the image of $I$ in $T_\p$ is minimally generated by a regular system of parameters, so this is equivalent to nonsingularity of $R$.
\end{proof}

 \section{Lifting the $\Gamma$-construction to mixed characteristic}\label{sec:Gamma}  
 
The hypothesis of Theorem~B in \S\ref{sec:intro} (Theorem~\ref{thm-B}) requires that the discrete valuation ring have an F-finite residue field.  
Let $R$ be essentially of finite type over a complete local $K$-algebra whose residue class field is $K$ of characteristic $p >0$.
In \cite{HH} a construction is given in terms of a $p$-base $\La$ for such a field $K$ of characteristic $p$ and a sufficiently
small but cofinite subset $\G$ of $\La$ to give a faithfully flat purely inseparable extension of $R$,
denoted $R\Ga$, such that certain properties of $R$ are preserved by the extension, e.g., being reduced, a domain, regular, or Gorenstein F-regular. See (6.3) -- (6.13) of \cite{HH}. Since $R \inj R\Ga$  is purely inseparable, the prime 
spectra of the two rings may be identified, and if $\G$ is sufficiently small the singular locus is preserved.   Note that
certain results of \cite{HH} stated for $K$-algebras of finite type over a complete local $K$-algebra extend at once to
the case of $K$-algebras essentially of finite type. Also note that the term ``$\Gamma$-construction" is not used in
\cite{HH}, but has become standard in a number of papers where it has been utilized.

In this section we discuss lifting
this construction to mixed characteristic, which broadens the applicability of Theorem~B.   Instead of working over a field $K$ of
positive prime characteristic, we work over a mixed characteristic unramified discrete valuation ring $(V, pV,K)$.  
We fix a $p$-base $\La$ for $K$ and a lifting $\tL$ of that $p$-base to $V$.  We consider a local ring $B$ containing $V$ complete with respect
to an ideal $\fb$ such that $pB + \fb$ is the maximal ideal of $B$ (for example, $B$ might be $V[[\vect x n]]$  whether or not $V$
is complete) and an algebra $R$ essentially of finite type over $B$.  
We construct  a faithfully flat extension $R\tGa$ of $R$ such that the defining ideal of the singular locus in $R$ extends to the defining ideal
of the singular locus in $R\tGa$  when $\tG$ is a sufficiently
small cofinite subset of $\tL$.  Moreover,  $R\tGa/pR\tGa \cong (R/pR)\Ga$.  

We first review
the $\G$-construction in characteristic $p$, and then discuss the mixed characteristic version.

\subsection{The $\Gamma$-construction in positive prime characteristic.} 

\begin{discussion}\label{pbase} Fix a field $K$ of positive prime characteristic $p$, and fix a $p$-base $\La$ for $K$.  Let $e$ denote an integer varying in $\N$.   
If $\vect \la n$ are elements of the $p$-base $\La$,
then $K^{p^e}[\vect \la n]$ has degree $p^{ne}$ over $K^{p^e}$, or, equivalently, $K[\la_1^{1/p^e}, \, \ldots, \la_n^{1/p^e}]$ has degree $p^{en}$ over $K$,
where the elements $\la_i^{1/p^e}$ are taken in a suitably large algebraic field extension of $K$.  It follows that if the $z_\la$ are
indeterminates indexed by $\La$,  then for any subset $\La_0 \inc \La$,  
$K[\la^{1/p^e}: \la \in \La_0] \cong K[z_\la: \la \in \La_0]/(z_\la^{p^e} - \la:\la \in \La_0)$.  
\end{discussion}

Let $K$ and $\La$ be as in Discussion~\ref{pbase}.  In the sequel, $\G$ will always denote a subset of $\La$, usually {\it cofinite} in $\La$, i.e.,  a subset such that $\La \smallsetminus \G$
is finite.  Let $\KGe$ denote the field extension $K[\g^{1/p^e}: \g \in \G]$.  If $(B, \m_B)$ is a complete local $K$-algebra with residue class field $K$,
let $B\Ge := \KGe \wh{\otimes}_K B$,   where for a field extension $K \to L$,  $L\wh{\otimes}_K B$ denotes the completion of $L \otimes_K B$ with
respect to $\m_B(L \otimes_K B)$.   If we write $B$ as a module-finite extension of $A = K[[\vect x d]]$,  then 
$A\Ge \cong K\Ge[[\vect x d]]$, $B\Ge \cong B \otimes_A  A\Ge$.  If $R$ is any ring essentially of finite type over
$B$,  we define $R\Ge = R \otimes_B B\Ge$.  $R$ may also be viewed as essentially of finite type over $A$,  and it
is also correct that $R\Ge = R \otimes_A A\Ge$.  Finally, we define 
$$R\Ga := \bigcup_e R\Ge.$$

We also have 
$$R\Ga \cong R \otimes_B B\Ga  \cong R \otimes_A A\Ga.$$
In characteristic $p >0$, we say that $R \inj S$ is {\it purely inseparable} if every element in $S$ has a $p^e\,$th
power in $R$ for some $e$, which may depend on the element.  This implies $\Spec(S) \cong \Spec(R)$:  the
unique prime of $S$ lying over $P$ in $R$ is the radical of $PS$.
Note that $$A\Ga = \bigcup_e K\Ge[[\vect x d]]$$
is regular, with maximal ideal generated by $\vect x d$, and  that it is purely inseparable over $A$.   It is shown in 
\cite{HH} that $A\Ga$ is F-finite, and it follows that $B\Ga$ and $R\Ga$ are F-finite as well. Since $A\Ga$ is
evidently faithfully flat over $A$,  it follows that $R\Ga$ is faithfully flat over $R$ in general.  We summarize
these facts as well as some other results proved in \cite{HH} in the next theorem.

\begin{notation}\label{not:G}  We let $(B, \m_B, K)$ be a complete local ring $K$-algebra, and $K$ be a field of positive prime characteristic $p$, so that $K \inc B$ is a coefficient field.
Let $\La$ be a fixed $p$-base for $K$, let $R$ be an algebra essentially of finite type over $B$,  and let
$\G$ be a varying cofinite subset of $\La$, so that $R\Ga := B\Ga \otimes_B R$ is defined. \end{notation}

\begin{theorem}\label{thm:Gprop}  Let notation be as in \ref{not:G}. Then:
\begin{enumerate}
\item $R\Ga = B\Ga \otimes_B R$ is an F-finite ring  that is a faithfully flat purely inseparable extension of $R$.  
\item If $(A, \m_A, K)$ is  a complete local $K$-algebra such that $A \to B$ is a local module-finite
homomorphism, then $A\Ga \otimes_A R \cong B\Ga \otimes_B R = R\Ga$. 
\item If $\Gamma \inc \Gamma_0$ are cofinite in $\La$,  then $R^\Gamma \to R^{\G_0}$ is faithfully
flat and purely inseparable.  
\item If $I$ is an ideal of $R$, then $(R/I)\Ga \cong R\Ga/IR\Ga$.  
\item Given finitely many prime or radical ideals in $R$, there exists $\Gamma_0$ cofinite in $\La$ such
that for all $\Gamma \inc \Gamma_0$ cofinite in $\Gamma$,   the expansions of these ideals to $R\Ga$
remain prime or reduced, respectively.  In particular, if $R$ is a domain or reduced,  one can choose  $\Gamma_0$
so that $R^\Gamma$ is a domain or reduced, respectively, for all $\Gamma \inc \Gamma_0$ cofinite in $\G_0$.  
\end {enumerate}
\end{theorem}

For (c), note that for complete local rings of the form $A = K[[\vect xn]]$, the faithful flatness of $A\Ga \to A^{\G_0}$
is clear,  and the general case follows by base change to $R$.  

We note that part (e) reduces to the case of one prime ideal, which may be taken to be (0), since a radical
ideal is a finite intersection of prime ideals, and the various choices of $\Gamma$ for individual primes may
be intersected.  For the proof part (e), see (6.13) of \cite{HH} (the extension from finite to essentially finite algebras
is immediate).

Consider a property $\cP$ of rings (e.g., Cohen-Macaulay, Gorenstein, or regular).  We are interested
in the existence of $\Gamma_0$ cofinite in $\La$ such that for all $\Gamma \inc \Gamma_0$ cofinite
in $\G$,  the locus of primes $\p$ where $R$ has property $\cP$ is the same
in $R$ as in $R\Ga$.  If the locus is closed and defined by $I \inc R$,  this is equivalent to
the assertion that the corresponding locus is defined by $IR\Ga$ in $R\Ga$.   

\begin{discussion}\label{flatloc} Recall that if one has a flat local map $(R,\m) \to (S, \n)$ then $S$ is \CM\ if and only if both $R$ and the closed fiber
$S/\m S$ are \CM,  in which case the type of $S$ is the product of the types of $R$ and $S/\m S$.  Moreover,
if $S$ is regular, then $R$ is regular; cf. \cite{Mat}.  \end{discussion} 

We now recall some important properties of the $\Gamma$-construction in positive characteristic.
The result below overlaps substantially with \cite{Has, HH, Mur}, though we include it for completeness.

\begin{theorem} Let notation be as in \ref{not:G}. Let $\cP$ be a property of Noetherian rings such that
\ben
\item $R$ has  $\cP$ if and only if all of its localizations have $\cP$.
\item If $R$ is essentially of finite type over an excellent local ring,  the locus where $R$ has $\cP$ is open.  
\item If $R$ to $S$ is a flat local homomorphism and $S$ has $\cP$ then $R$ has $\cP$.
\item If $R$ to $S$ is a flat, purely inseparable, local homomorphism such that $R$ has $\cP$ and the closed fiber is a field, then $S$ has $\cP$.
\een
Then there exists a cofinite subset $\G_0$ of $\La$ such that for all cofinite subsets
$\G \inc \G_0$,  the locus where $R$ has property $\cP$  (and, consequently, the locus where $R$ does not have property $\cP$)
is the same as the corresponding locus for $R\Ga$ under the identification $\Spec(R) \cong \Spec(R\Ga)$. \smallskip

Hence, if the locus where $R$ does not have $\cP$ is Zariski closed with defining ideal~$I$,  for all $\G \subseteq \G_0$ cofinite in
$\La$, the locus where $R\Ga$ does not have $\cP$ is closed and defined by $IR\Ga$.

In particular, this is the case if $\cP$ is the property of being any of the following:
\beni
\item  regular
\item Cohen-Macaulay
\item Cohen-Macaulay of type at most $h$
\item Gorenstein
\item Gorenstein and F-regular.
\een
\end{theorem}
\begin{proof}  As $\G$ cofinite in $\La$ decreases, the locus where $R\Ga$ has $\cP$ is ascending by Theorem~\ref{thm:Gprop}(c).  Since open sets in $\Spec(R)$ have ACC,  we can choose $\Gamma_0$ for which this locus is
maximal.  We claim that this choice of $\Gamma_0$ has the required property.  If not, there is some prime $\p$ in $R$ corresponding
to $\q$ in $R^{\G_0}$ such that $R_{\p}$ has property $\cP$ by $R^{\G_0}_Q$ does not.  But we may choose $\Gamma_1$ cofinite
in $\Gamma_0$  such that $\q_1  = \p R^{\Gamma_1}$ is prime.  This implies that the closed fiber of the flat local map
$R_\p \to (R^{\G_1})_{\q_1}$ is a field, since $\q_1 = \p R^{\G_1}$.   Thus, the set of points of $\Spec(R^{\G_1})$ where  $\cP$ holds
has increased strictly to contain the point $Q_1$ corresponding to $P$, contradicting the maximality of the $\cP$ locus corresponding
to $\Gamma_0$.  

The fact that we may apply this result in cases (i)--(iv) follows from Discussion~\ref{flatloc}. In the case of (v) we need check (4).
Note that in the excellent Gorenstein case, weakly F-regular and F-regular are equivalent, and in the local case it suffices to
check that the ideal $I$ generated by one \sop\ $\vect fd$ is tightly closed.  Let the image of $u$ be a socle generator 
in $R/(\vect fd)$.  Since the closed fiber is a field, the image of $u$ is a socle generator in  $S/IS$ as well.  
Let $c$ be a test element in $S$.  After replacing
$c$ by $c^{p^e}$ we may assume that $c \in R$.  If $u$ is in the tight closure of $IS$ then $cu^q \in (IS)^{[q]} = I^{[q]}S$ for
all $q \gg 0$,  and so $cu^q \in I^{[q]}S \cap R = I^{[q]}$ for all $q \gg 0$,  contradicting that $R$ is weakly F-regular.  \end{proof}

\subsection{The $\wt{\Gamma}$-construction in mixed characteristic.} 

\begin{notation} Let $(V, pV, K)$ be a discrete valuation ring of mixed characteristic~$p$. Let $\Lambda$ be a $p$-base for $K$, and $\wt{\Lambda}$ be a fixed lifting to $V$. 
 We denote the by $\wt{\lambda}$ the element lifting $\lambda \in \La$. 
For each subset $\Gamma$ of 
$\Lambda$ we have a corresponding set $\wt{\Gamma}$ in $\wt{\Lambda}$.  We shall eventually restrict $\Gamma$ to be cofinite in $\La$, but this is not needed in the definitions. 
We assume that we have, for every $e \geq 1$,  a family of indeterminates $Z_{\wt{\la},e}$ indexed bijectively by the elements of $\wt{\La}$.  Of course, they are also in bijective correspondence with the
elements of $\La$.  We define $V\Gte$ to be the ring ${V[Z_{\gt,e}: \gt \in \wt{\G}]}/{(Z_{\gt,e}^{p^e} - \gt: \gt \in \wt{\G})}$.  Note that there is a $V$-algebra injection of $V\Gte \inj V\Gt_{e+1}$ 
that sends the image of $Z_{\gt, e}$ to the image of $Z_{\gt, e+1}^p$.  We define $V\Gt := \bigcup_e V\Gte$ using these injections.
 \end{notation}
 
 We first note:
 
 \begin{proposition}  All of the rings $V\Gte$ as well as $V\Gt$ are discrete valuation rings with maximal ideal generated by~$p$. Moreover, $V\Gte/pV\Gte \cong K\Ge$, and $V\Gt /pV\Gt \cong K\Ga$.
 For $e' > e$, and $\G' \supset \G$,   The inclusion maps among the rings 
 $$ (\dagger) \quad V \inj V\Gte \inj V_{e'}^{\wt{\G'}} \inj V^{\wt{\G'}} \inj V^{\wt{\La}}$$ are all free local maps.
 \end{proposition}
 \begin{proof} Since the extension $V \to V\Gte$ is integral, all maximal ideals of 
 $V\Gte$ must lie over $pV$. 

 Modulo $pV\Gte$,  we obtain $K[Z_{\gt}]/(Z_{\gt}^{p^e} - \g: \g \in \G)$.  Since the extensions of $K$ obtained by adjoining $p^e\,$th roots
 of a $p$-base are linearly disjoint,  the quotient mod $pV\Gte$ is the field $K\Ge$.  To show this is a discrete valuation ring with maximal ideal generated by $p$,  it suffices to show
 this for finite subsets of $\wt{\Gamma}$, and take a direct limit. But in the case where where we adjoin $p^e\,$th roots for a finite subset of $\G$,   we obtain a module-finite extension, free over $V$, 
 which is local with maximal ideal $(p)$, and $p$ is not nilpotent. These hypotheses suffice to guarantee that the extension ring is a discrete valuation ring.   Likewise, by a direct limit argument, $V\Gt$ is a discrete valuation ring with maximal ideal $pV\Gt$ and residue class field $K\Ga$.  
 
 The freeness assertions for $V^{\wt{\La}}$ over $V$ follows because one has a free basis consisting of all monomials $\wt{\la}_{i_1}^{\alpha_1} \cdots \wt{\la}_{i_s}^{\alpha_s}$
 where $\la_{i_1}, \, \ldots,\, \la_{i_s}$ are distinct, $s$ varies, and the $\alpha_i$ are positive rational numbers less than 1 whose denominators are powers of~$p$.  
 One gets a basis for each $V\Gt$ (respectively, $V\Gte$) by using the part of the basis involving only monomials with $\la_{i_j} \in \Gamma$ (respectively, and the denominators in the exponents 
  are most $p^e$).  Descriptions of the free bases in the remaining cases are left to the reader.  
 \end{proof}

 \begin{construction}\label{con:Gt}
 Let  $A := V[[\vect x d]]$.   We let $$A\Gte := V\Gte \wh{\otimes}_V^{\ux} A \cong V\Gte[[\vect xd]],$$
 where $\wh{\otimes}^{\ux}$ is completion with respect to the expansion of 
 the ideal $(\vect x d)$.  We then let $A\Gt := \bigcup_e A\Gte$.  
 
 Now suppose that $R$ is essentially of finite type over $A$.  We define $R\Gte := A\Gte \otimes_A R$,  and $R\Gt := A\Gt \otimes_A R \cong \bigcup_e R\Gte$,  since
 $\otimes_A$ commutes with direct limit.  
 
 We also  describe a more invariant formulation that does not directly utilize formal power series.  
 
  Let $B$ denote a local $V$-algebra with residue class field $K$ such
 that $B$ is complete with respect to an ideal $\fb$ such that $p + \fb$ is the maximal ideal of $B$. 
 
  Then there is a $V$-algebra surjection $A = V[[\vect x d]] \surj  B$
 such that the $x_j$ map to generators of $\fB$.  Then  we can construct $B\Gte$ using the fact that $B$ is an image of $A$,  or directly, as the
 completion of $V\Gte \otimes_V B$ with respect to the expansion of $\fB$,  which we denote $V\Gte \wh{\otimes}^{\fb}_V B$.   
 
 Likewise, we may construct
 $B\Gt$ by thinking of $B$ as an $A$-algebra, or directly as $\bigcup_e  B\Gte$.   Now, if $R$ is essentially of finite type over $B$, it does not matter whether
 we take $R\Gte$ (respectively $R\Gt$) as $B\Gte \otimes _B R$ (respectively,  $B\Gt \otimes_V R$ ) or as $A\Gte \otimes_A R$  (respectively,  $A\Gt \otimes_V R$).
 It also follows that if we have two choices of   $B$,  say $B \to C$ with the map local and $C$ module-finite over the image of $B$, and $R$ is essentially of finite
 type over $C$ and, hence, $B$,  we may  calculate $R\Gte$ or $R\Gt$ using either $B$ or~$C$.   \end{construction}

\begin{theorem} Let notation be as in Construction~\ref{con:Gt}. For all sufficiently small cofinite $\wt{\Gamma} \inc \wt{\La}$,  the defining ideal of the singular locus in $R$ expands to 
the defining ideal of the singular locus in $R\Gt$.  \end{theorem}
 \begin{proof}  Choose $\Gamma$ so that the set of primes $\q$ containing $p$ where $R_\q$ and $R\Gt_{\wt{\q}}$ are both regular is maximal.  Note that we may identify
 $V_R(p) \cong  \Spec(R/pR)$ with $V_{R\Gt}(p) \cong \Spec(R\Gt/(p)) \cong (R/pR)\Ga$:   here $\wt{\q}$ denotes the unique prime of
 $R\Gt$ lying over $\q$.  One can do this since the open sets in $\Spec(R/pR)$ have ACC. 
 
 For primes of $R$ containing  $p$ the maximality of $\Gamma$ implies that the regular locus in $R$ is identified with the regular locus of primes in $R\Gt$ containing $p$:
 if this fails for $\q$, one may decrease $\Gamma$ so that $\q R\Gt = \wt{\q}$ (the issues may be considered modulo $p$, where the statement follows from the corresponding result
 for the $\Gamma$-construction).   Hence, the regular loci for primes containing $p$
 agree for this $\Gamma$ or any other smaller cofinite choice of $\Gamma$.     
 At primes not containing $p$, one may localize both rings at the element $p$.  The map is then geometrically regular, and one has regularity at a prime of the target if and only if
 one has regularity at its contraction. \end{proof}

We conclude by applying the mixed characteristic $\Gamma$-construction to give a regularity criterion without any F-finiteness hypotheses on the residue field. The following is Theorem~C from the introduction.

\begin{corollary}\label{cor-thmC}
	Let $V$ be an unramified discrete valuation ring of mixed characteristic, and let $(R,\m,k)$ be essentially of finite type over a complete local $V$-algebra with $p\in \m$. We retain the notation of Construction~\ref{con:Gt}.
	
	The ring $R$ is regular if and only if for any sufficiently small cofinite subset $\wt{\Gamma} \subseteq \wt{\Lambda}$, $\wt{\Omega}_{R\tGa|(\Z,\delta)}$ is free of rank $\dim(R)+\log_p[k(R^{\wt{\Gamma}}):k(R^{\wt{\Gamma}})^p]$, where $k(R^{\wt{\Gamma}})$ denotes the residue field of $R^{\wt{\Gamma}}$.
\end{corollary}

\section*{Acknowledgments} The authors thank Elo\'isa Grifo and Zhan Jiang for helpful comments on a draft of this article. We also thank the referee for many valuable comments.

	\quad\bigskip
	
	
	
\end{document}